\documentclass[12pt,a4paper]{amsart}
\usepackage[top=40mm, bottom=35mm, left=35mm, right=35mm]{geometry}
\usepackage{mathptmx}
\usepackage{enumerate}
\usepackage{verbatim}
\usepackage{url}

\usepackage[colorlinks=true,citecolor=blue]{hyperref}

\theoremstyle{plain}
\newtheorem{thm}{Theorem}[section]
\newtheorem{lem}[thm]{Lemma}
\newtheorem{prop}[thm]{Proposition}
\newtheorem{cor}[thm]{Corollary}

\theoremstyle{definition}
\newtheorem{de}[thm]{Definition}
\newtheorem{rem}[thm]{Remark}

\newcommand{\Z}{{\mathbb{Z}}}

\newcommand{\N}{\mathbb{N}}

\DeclareMathOperator{\diam}{diam}
\DeclareMathOperator{\supp}{supp}
\DeclareMathOperator{\orb}{orb}

\numberwithin{equation}{section}

\begin{document}
\title[When is a dynamical system mean sensitive?]{When is a dynamical
system mean sensitive?}
\author[F. Garc\'ia-Ramos]{Felipe~Garc\'ia-Ramos}
\address[F. Garc\'{\i}a-Ramos]{Instituto de Fisica, Universidad Autonoma de
San Luis Potosi, Manuel Nava 6, SLP, Mexico 78290 and Catedras CONACyT}
\email{felipegra@yahoo.com}
\author[J.~Li]{Jie Li}
\address[J. Li]{School of Mathematics and Statistics,
Jiangsu Normal University, Xuzhou, Jiangsu, 221116, P.R. China --\and--
Academy of Mathematics and Systems Science, Chinese Academy of Sciences,
Beijing, 100190, P.R. China}
\email{jiel0516@mail.ustc.edu.cn}
\author[R.~Zhang]{Ruifeng Zhang}
\address[R.~Zhang]{School of Mathematics, Hefei University of Technology,
Hefei, Anhui, 230009, P.R. China}
\email{rfzhang@mail.ustc.edu.cn}
\subjclass[2010]{54H20, 37B10, 54B20, 37B05}
\keywords{Mean sensitivity, Banach-mean sensitivity, diam-mean sensitivity,
mean equicontinuity, Banach-mean equicontinuity, diam-mean equicontinuity,
hyperspace}

\begin{abstract}
This article is devoted to study which conditions imply that a topological
dynamical system is mean sensitive and which do not. Among other things we
show that every uniquely ergodic, mixing system with positive entropy is
mean sensitive. On the other hand we provide an example of a transitive
system which is cofinitely sensitive or Devaney chaotic with positive
entropy but fails to be mean sensitive.

%
%

As applications of our theory and examples, we negatively answer an open
question regarding equicontinuity/sensitivity dichotomies raised by Tu, we
introduce and present results of locally mean equicontinuous systems and we
show that mean sensitivity of the induced hyperspace does not imply that of
the phase space.
\end{abstract}

\maketitle
\tableofcontents

\section{Introduction}

\label{sect:Intro}

A pair $(X,T)$ is called a \textit{topological dynamical system} (or simply
\textit{t.d.s.}) if $X$ is a compact metric space with metric $d$ and $T\colon
X\rightarrow X$ is a continuous map. A Borel probability measure $\mu $ in $%
X $ is \textit{ergodic} if it is invariant under $T$ and every invariant set
has measure 0 or 1.

A question of interest in topological dynamical systems is when orbits from
nearby points deviate. \textit{Sensitive dependence on initial conditions}
(or briefly \textit{Sensitivity}), appeared as the first explicit
mathematical definition to describe this turbulent behavior \cite{Rue77}.
According to the work of Auslander and Yorke \cite{AY80}, we say that a
t.d.s. $(X,T)$ is \textit{sensitive} if there exists $\delta>0$ such that
for any non-empty open subset $U\subset X$, there exist $x,y\in U$ and $n\in%
\mathbb{N}$ such that $d(T^{n}x,T^{n}y)>\delta$. Following from \cite[%
Theorem 3.4]{AK03} we have the following proposition.

\begin{prop}
\label{prop:Sen} A t.d.s. $(X, T)$ is sensitive if and only if there exists
a $\delta>0$ such that for any non-empty open subset $U\subset X$ one of the
following holds:

\begin{enumerate}
\item \label{prop:Sen:1} $N_{T}(U,\delta):=\{n\in{\mathbb{Z}}_{+}: \diam%
(T^{n}(U))>\delta\}$ is infinite, where $\diam(\cdot)$ denotes the diameter
of the set;

\item \label{prop:Sen:2} $\limsup_{n\to\infty}d(T^{n}x,T^{n}y) >\delta$ for
some $x,y\in U$;

\item \label{prop:Sen:3} $\limsup_{n\to\infty}\diam(T^{n}U)>\delta$.
\end{enumerate}
\end{prop}

Sensitivity as a form of chaos is very weak. For example the Sturmian
subshift, which is considered as a very rigid system, is sensitive.

Stronger forms of sensitivity have also been studied by adding extra
requirements on the set $N_{T}(U,\delta )$ in \eqref{prop:Sen:1}. This idea
was carried out by Moothathu in \cite{Moo07}. In particular, he introduced
the notion of cofinite sensitivity. That is, a t.d.s. $(X,T)$ is \textit{%
cofinitely sensitive} if there is a $\delta >0$ such that for any non-empty
open subset $U\subset X$, $N_{T}(U,\delta )$ is cofinite, i.e. $\diam%
(T^{i}U)\leq \delta $ for only finitely many times.

Another interesting approach to strengthen sensitivity (due to its
connections with ergodic theory) is to replace the upper limits in %
\eqref{prop:Sen:2} and \eqref{prop:Sen:3} by upper average limits (Cesaro
and Banach averages) \cite{LTY13, GM14}.
Strictly speaking:

\begin{de}
\label{def:Mean-Sen} A t.d.s. $(X,T)$ is \textit{mean sensitive} (or \textit{%
Banach-mean sensitive}) if there is a $\delta>0$ such that for any
neighbourhood $U$ of $X$, there are $x,y\in U$ such that
\begin{equation*}
\limsup_{n\rightarrow\infty}\frac{1}{n}\sum_{i=0}^{n-1}d(T^{i}x,T^{i}y)>%
\delta\ (\text{ or }\limsup_{N-M\rightarrow\infty}\frac{1}{N-M}\sum
_{i=M}^{N-1}d(T^{i}x,T^{i}y)>\delta).
\end{equation*}
A t.d.s. $(X,T)$ is \textit{diam-mean sensitive} if there is a $\delta>0$
such that for every neighbourhood $U$ of $X$,
\begin{equation*}
\limsup_{n\rightarrow\infty}\frac{1}{n} \sum_{i=0}^{n-1}\diam(T^{i}U)>\delta.
\end{equation*}
\end{de}

When considering the opposite side of sensitivity (resp. mean sensitivity,
Banach-mean sensitivity, diam-mean sensitivity), the notions of
equicontinuous (resp. mean equicontinuous, Banach-mean equicontinuity,
diam-mean equicontinuous) points and systems appear correspondingly ( \cite%
{GW93, LTY13, GM14} respectively); see Section \ref{sect:preliminaries} for definitions.

\smallskip Mean sensitivity/equicontinuity has been studied recently because
it turned out to be a useful concept to describe/characterize ergodic
theoretic properties like when a system has discrete spectrum
and when its maximal equicontinuous factor is an isomorphism with
topological notions. 
\cite{GM14, DG15, LTY13}. The main concern of this paper is to explore which
conditions imply mean sensitivity and which do not.

\medskip The following implications follow from the definitions
\begin{equation*}
\begin{array}{cccc}
& \text{mean sensitivity}\Rightarrow & \text{Banach-mean sensitivity} &  \\
& \Downarrow &  &  \\
\text{cofinite sensitivity}\Rightarrow & \text{diam-mean sensitivity} &  &
\end{array}%
\end{equation*}%
From examples in \cite{GM14,LTY13} we know each implication is strict. \
Note that the first row represents `point' forms of sensitivity where points
are used in the definition 
and in the second row we write `diameter' forms of sensitivity when the
sensitivity is measured with the diameter of the orbit of open sets.
Motivated by this we have the following question: does cofinite sensitivity
(which is the strongest form of `diameter' sensitivity) imply Banach-mean
sensitivity (the weakest form of `point' sensitivity)? This question has a
negative answer.

\begin{thm}
\label{thm:Diam-Mean-Sen not Mean-Sen} There exists a transitive cofinitely
sensitive t.d.s. $(X,T)$ which is Banach-mean equicontinuous and hence not
Banach-mean sensitive.
\end{thm}

Other popular forms of chaos are positive entropy, mixing, Li-Yorke chaos
and Devaney chaos. A system is \textit{Devaney chaotic} if it is transitive
and the set of periodic points is dense. For the definition of Li-Yorke
chaos see \cite{BGKM} . Considerable time has been spent figuring out which
conditions of chaos are stronger than others. For example it is known that
positive topological entropy, Devaney chaos and weak mixing implies Li-Yorke
chaos \cite{BGKM, HY, Iwanik91}.

In this paper we show that:

\begin{thm}
\label{thm:Dich-E-system} There exists a Devaney chaotic t.d.s. with
positive topological entropy that is almost mean equicontinuous and hence not
mean sensitive.
\end{thm}

Note that in \cite{LTY13} a transitive non mean sensitive system with
positive entropy was constructed, and also there exist mixing systems that
are not mean sensitive.

\medskip For clarity, we summarise the conditions that do not imply mean
sensitivity:

\begin{itemize}
\item Topological mixing \cite{LTY13}.

\item Minimality plus maximal equicontinuous factor is not 1-1 \cite{DG15}.

\item Cofinite sensitivity (Theorem \ref{thm:Diam-Mean-Sen not Mean-Sen}).

\item Devaney chaos plus positive topological entropy (Theorem \ref%
{thm:Dich-E-system}).

\item Positive entropy plus unique ergodicity (Theorem \ref{unpos}).
\end{itemize}

\medskip For the positive side, the only known conditions that are stronger
than mean sensitivity are

$\ \ \ \ \ \bullet $\ minimality plus positive topological entropy \cite{LTY13, GM14}.

$\ \ \ \ \ \bullet $\ ergodic measure with full support and not purely discrete
spectrum.

We add other properties to the list. The reader is referred to Section \ref%
{sect:Positive-Cond} for details.

\begin{thm}
The following conditions imply mean sensitivity:

\begin{itemize}
\item Topologically mixing, unique ergodicity and positive topological
entropy.

\item Minimality and topological weak mixing.

\item Transitivity and shadowing property with positive topological entropy.
\end{itemize}
\end{thm}

There are several applications to the theory and counterexamples we
constructed in the following areas: (1) equicontinuity/sensitivity
dichotomies, (2) locally mean equicontinuous systems and (3) dynamical
theory of hyperspaces.

\medskip As for dichotomies, it was shown in \cite{AY80} that a minimal
t.d.s. is either equicontinuous or sensitive. This result was generalized by
showing that every $E$-system is either equicontinuous or sensitive, see
\cite[Theorem 1.3]{GW93} or \cite[Theorem 4.6]{HY02}. Recall that t.d.s. is
an $E$\textit{-system} if it is transitive and there exists an invariant
measure with full support.

For mean notions we know that a minimal t.d.s. is either mean equicontinuous
or mean sensitive \cite{LTY13, GM14}. In \cite[Question 3.7.3]{Tu14} Tu
wondered if each $E$-system is either mean sensitive or mean equicontinuous.
Since every Devaney chaotic t.d.s. is an $E$-system and mean equicontinuous
systems have zero entropy (\cite{LTY13, GM14}) we obtain for the above
question a negative answer as a corollary of Theorem \ref{thm:Dich-E-system}.

\begin{cor}
\label{Dich-E-Cor}There exist $E$-systems which are neither mean sensitive
nor mean equicontinuous.
\end{cor}

Note that in \cite[Theorem 13]{GM14} it was shown that a strongly transitive
system (this property is stronger than $E$-system) is either mean sensitive
or mean equicontinuous. Here we show that systems which are transitive and
have dense minimal points (so called \textit{$M$-systems}) are either Banach
mean equicontinuous or Banach-mean sensitive (Corollary \ref{cor:Dich-M-system}%
), and systems which are transitive and have shadowing property (stronger
than $M$-systems) are either mean sensitive or mean equicontinuous
(Corollary \ref{cor:Trans+ShadowProp}).

\medskip In \cite{GW00} Glasner and Weiss introduced the concept of local
equicontinuity. A t.d.s. $(X,T)$ is \textit{locally equicontinuous (LE) }if
for every $x\in X$ we have that $\overline{\orb(x,T)}$ is almost
equicontinuous. Inspired by this, we say that a t.d.s. $(X,T)$ is \textit{%
locally mean equicontinuous (LME) }if for every $x\in X$ we have that $%
\overline{\orb(x,T)}$ is almost mean equicontinuous. In Section \ref%
{sect:LME} we prove an application of Theorem \ref{thm:Dich-E-system} by
showing the following :

\begin{thm}
\label{thm:LME} Similarly to LE systems LME systems have zero topological
entropy. Nonetheless, contrary to LE systems, ergodic measures on LME
systems may be supported on non-minimal subsystems.
\end{thm}

\medskip In `point' forms of sensitivity (like mean sensitivity) we care
about the behaviour of points. In `diameter' forms of sensitivity we pay
attention to the behaviour of sets. There is a general way to study the
behaviour of sets: the study on hyperspatial dynamical systems. Given a
t.d.s. $(X,T)$, we can naturally induce a t.d.s. $(K(X),T_{K})$, where $K(X)$
is the hyperspace consisting of all non-empty closed subsets of $X$ and
endowed with the Hausdorff metric (cf. \cite{Nad92}). Bauer and Sigmund in
\cite{BS75} initiated a systematic study on the connections between dynamics
of $(X,T)$ and $(K(X),T_{K})$. In particular, they proved that 
$(X,T)$ is weakly mixing if and only if $K(X)$ is weakly mixing. Later,
Banks \cite{Ban05} showed that on $(K(X),T_{K})$ weak mixing is equivalent
to transitivity. Very recently, Wu et al. \cite{WWC15} summarized the
connections on $\mathcal{F}$-sensitivity (where $\mathcal{F}$ is a
Furstenberg family), and particularly they showed that $\mathcal{F}$%
-sensitivity of $(K(X),T_{K})$ implies that of $(X,T)$, and the converse is
also true when additionally $\mathcal{F}$ is a filter.
We refer the interested reader to \cite{LOW16, LOYZ15, LYY13} for further
connections.

In this paper we will study how mean forms of sensitivity and equicontinuity
behave on hyperspaces. Following definition and a bit of work (similarly as
in \cite[Corollary 1]{WWC15}) it is not hard to see that if $(K(X),T_{K})$ is
diam-mean sensitive then so is $(X,T)$. Nonetheless, using the example in
the proof of Theorem \ref{thm:Diam-Mean-Sen not Mean-Sen} we show that for
mean sensitivity this does not happen.

\begin{thm}
\label{thm:K(X)-Mean-Sen-X-mean-Equi} There exists a transitive t.d.s. $%
(X,T) $ with zero topological entropy that is Banach-mean equicontinuous
(thus not mean sensitive) such that $(K(X),T_{K})$ is mean sensitive and has
positive topological entropy .
\end{thm}

Note that in \cite{GW95} Glasner and Weiss constructed the first example of
a t.d.s. $(X,T)$ that is minimal and has zero topological entropy but $%
(K(X),T_{K})$ has positive topological entropy. These results indicate that
induced dynamics on hyperspaces admit more complicated behaviours than
dynamics on the phase space.

We also show that on the hyperspace mean equicontinuity and diam-mean
equicontinuity are equivalent (Corollary \ref%
{cor:K(X)-Diam-MeanEqui=MeanEqui}); for mean sensitivity we have the
following result.

\begin{thm}
\label{thm:KK(X)-Mean-Sen=Diam-Mean-Sen} Let $(X,T)$ be a t.d.s. and $%
(K(X),T_K)$ be its induced hyperspatial t.d.s. Consider the following
statements:

\begin{enumerate}
\item \label{thm:Mean-Sen=Diam-Mean-Sen:1} $(K(X),T_{K})$ is diam-mean
sensitive,

\item \label{thm:Mean-Sen=Diam-Mean-Sen:2} $(K(K(X)),T_{K})$ is diam-mean
sensitive,

\item \label{thm:Mean-Sen=Diam-Mean-Sen:3} $(K(X),T_{K})$ is mean sensitive,

\item \label{thm:Mean-Sen=Diam-Mean-Sen:4} $(K(K(X)),T_{K})$ is mean
sensitive.
\end{enumerate}

Then \eqref{thm:Mean-Sen=Diam-Mean-Sen:1} and %
\eqref{thm:Mean-Sen=Diam-Mean-Sen:2} are equivalent. If additionally $(X,T)$
is weakly mixing, \eqref{thm:Mean-Sen=Diam-Mean-Sen:1}--%
\eqref{thm:Mean-Sen=Diam-Mean-Sen:4} are all equivalent.
\end{thm}

\medskip \textbf{Acknowledgments}

A part of this paper was done when the second author visited the Chinese
University of Hong Kong. He acknowledges the hospitality of the Mathematics
Department at the University. The authors also thank Jian Li, Siming Tu and
Xiangdong Ye for very useful suggestions and valuable comments. The authors
also thank the referee for   valuable suggestions and pointing out Remark 2.1 and
Remark 2.2 to us.

During the period of research Felipe Garcia Ramos was supported by NSERC,
IMPA and CAPES, Jie Li (corresponding author) was supported by NSF of Jiangsu Province (BK20170225), China Postdoctoral Science Foundation (2017M611026) and NNSF of China (11371339, 11571335),
and Ruifeng Zhang was supported by NNSF of China (11171320, 11671094).

\section{Preliminaries}

\label{sect:preliminaries}

Throughout this paper, we denote by ${\mathbb{Z}}_{+}$, and $\mathbb{N}$ the
sets of non-negative integers and natural numbers respectively.

\medskip \textbf{$\bullet$ Subsets of ${\mathbb{Z}}_{+}$ } \medskip

Let $F$ be a subset of ${\mathbb{Z}}_{+}$. The \textit{density} and \textit{%
upper Banach density} of $F$ are defined by
\begin{equation*}
\overline{D}(F)=\lim \sup_{n\rightarrow \infty }\frac{\#\{F\cap \lbrack
0,n-1]\}}{n}
\end{equation*}%
and
\begin{equation*}
BD^{\ast }(F)=\limsup_{N-M\rightarrow \infty }\frac{\#\{F\cap \lbrack
M,N-1]\}}{N-M}=\limsup_{n\rightarrow \infty }\left\{ \sup_{N-M=n}\frac{%
\#\{F\cap \lbrack M,N-1]\}}{n}\right\} ,
\end{equation*}%
where $\#\{\cdot \}$ denotes the cardinality of the set. It is clear that $%
\overline{D}(F)\leq BD^{\ast }(F)$ for any $F\subset {\mathbb{Z}}_{+}$.

\medskip We say that a subset $F\subset{\mathbb{Z}}_{+}$ is \textit{syndetic}
if there is an $n\in\mathbb{N}$ such that $F\cap\{m,m+1,\ldots,m+n\}\neq
\emptyset$ for any $m\in{\mathbb{Z}}_{+}$; and is \textit{thick} if there
are $n_{i}\to\infty$ such that $F\supset\cup_{i=1}^{\infty}\{n_{i},
n_{i}+1,\ldots, n_{i}+i\}$.

\medskip \textbf{$\bullet$Topological dynamics} \medskip

Let $(X,T)$ be a t.d.s., $x\in X$ and $U$ a neighbourhood of $x$. We denote
the orbit of $x$ by $\orb(x,T)=\{x,Tx,\ldots \}$, its orbit closure by $%
\overline{\orb(x,T)}$, and the return times of $x$ to $U$ as $N(x,U)=\{n\in {%
\mathbb{Z}}_{+}\colon T^{n}x\in U\}.$

We say that a point $x$ is \textit{periodic} if $T^n x=x$ for some $n\in \Z_+$; is
\textit{transitive} if $\overline{\orb(x,T)}=X;$ is \textit{recurrent} if $%
N(x,U)$ is non-empty for any neighbourhood $U$ of $x$ and is \textit{minimal}
if $N(x,U)$ is syndetic for any neighbourhood $U$ of $x$.

A t.d.s. $(X,T)$ is a \textit{transitive system} if for any non-empty open
sets $U,V\subset X$ there exists $n\in \mathbb{N}$ such that $T^{n}U\cap
V\neq \emptyset $; is \textit{weakly mixing} if the product t.d.s. $X\times
X $ is transitive; and is a \textit{minimal system} if every point of $X$ is
transitive. It is well known that if $x$ is recurrent then $\overline{\orb%
(x,T)}$ is transitive.

Given a t.d.s. $(X,T)$, we denote $M(X,T)$ by the collection of all $T$%
-invariant Borel probability measures on $X$. It is well known that $M(X,T)$
is always non-empty. Let $\mu \in M(X,T)$. We define its \textit{support} by
$\supp(\mu )=\{x\in X\colon \mu (U)>0\text{ for any neighbourhood }U\text{ of
}x\}$.

We say that a t.d.s. $(X,T)$ is

\begin{enumerate}[$\ \ \ \ \cdot$]
\item  \textit{Devaney chaotic} if it is transitive and the set of periodic
points is dense;

\item an \textit{$M$-system} if it is transitive and the set of minimal
points is dense;

\item an \textit{$E$-system} if it is transitive and there exists an
invariant measure with full support.
\end{enumerate}

Every Devaney chaotic system is an $M$-system and each $M$-system is an $E$%
-system.

Let $(X,T)$ and $(Y,S)$ be two t.d.s. and $\pi \colon X\rightarrow Y$ a continuous function.
We say that $\pi $ is a \textit{factor map} if $\pi $ is surjective
and satisfies that $\pi \circ T=S\circ \pi $. In this case we say $X$ as an \textit{extension}
of $Y$ or $Y$ as a \textit{factor} of $X$. The factor $\pi$ is said to be \textit{almost one to one} if there exists a residual subset (i.e. contains a dense $G_{\delta }$ set) $G$ such that $\pi^{-1}(\pi(x))=\{x\}$ for all $x\in G$.

We refer the reader not familiar with topological entropy to the textbook
\cite{Wal82}.

\medskip \textbf{$\bullet $ Various forms of equicontinuity and sensitivity}
\medskip

We have defined various forms of sensitivity in Section \ref{sect:Intro}.

Let $(X,T)$ be a t.d.s. and $x\in X$. We say that the point $x\in X$ is

\begin{enumerate}[$\ \ \ \ \cdot$]
\item an \textit{(ordinary) equicontinuous point} if for any $\varepsilon>0$ there is a
$\delta>0$ such that for every $y\in X$ with $d(x,y)<\delta$, we have $%
d(T^{n} x, T^{n} y)<\varepsilon$ for all $n\in{\mathbb{Z}}_{+}$;

\item a \textit{mean equicontinuous point} if for any $\varepsilon>0$ there
is a $\delta>0$ such that for every $y\in X$ with $d(x,y)<\delta$, $%
\limsup_{n\to\infty} \frac{1}{n}\sum_{i=0}^{n-1}d(T^{i} x,T^{i}
y)<\varepsilon$;

\item a \textit{Banach-mean equicontinuous point} if for any $\varepsilon>0$
there is a $\delta>0$ such that for every $y\in X$ with $d(x,y)<\delta$, $%
\limsup_{N-M\to\infty}\frac{1}{N-M}\sum_{i=M}^{N-1}d(T^{i} x,T^{i}
y)<\varepsilon$;

\item a \textit{diam-mean equicontinuous point} if for any $\varepsilon>0$
there is a neighbourhood $U$ of $x$ such that $\limsup_{n\to\infty}\frac{1}{n}%
\sum_{i=0}^{n-1}\diam(T^{i} U)<\varepsilon$.
\end{enumerate}

A t.d.s. $(X,T)$ is \textit{$Q$-equicontinuous} (where $Q=$ ordinary or mean or Banach-mean or diam-mean) if all points in $X$ are $Q$-equicontinuous; and
is \textit{almost $Q$-equicontinuous} if the set of $Q$-equicontinuous points is residual. 

\begin{rem}\label{rem:equi}
\begin{enumerate}[(1)]
\item\label{rem:equi:1}
When every point is $Q$-equicontinuous (where $Q=$ ordinary or mean or Banach-mean or diam-mean), by compactness it is easy to check that the $\delta$ is independent of the choice of $x$.

\item\label{rem:equi:2}
With the usual triangle inequality argument, we can make clear the difference between mean equicontinuity and diam-mean equicontinuity. I.e.,

$\cdot $ $x\in X$ is a mean equicontinuous point if and only if for every $%
\varepsilon >0$ there exists $\delta >0$ such that $\sup_{y_{1},y_{2}\in
U}\lim \sup_{n\rightarrow \infty }\frac{1}{n}%
\sum_{i=0}^{n-1}d(T^{i}y_{1},T^{i}y_{2})<\varepsilon $, where $U=\left\{
z:d(z,x)<\delta \right\}$;

$\cdot $ $x\in X$ is a diam-mean equicontinuity point if and only if for every $%
\varepsilon >0$ there exists $\delta >0$ such that $\lim \sup_{n\rightarrow
\infty }\sup_{y_{1},y_{2}\in U}\frac{1}{n}%
\sum_{i=0}^{n-1}d(T^{i}y_{1},T^{i}y_{2})<\varepsilon $, where $U=\left\{
z:d(z,x)<\delta \right\}$.

\item\label{rem:equi:3}
According to Fomin \cite{Fom51}, a t.d.s. $(X,T)$ is \textit{mean-L-stable} if for every $\varepsilon >0$ there exists $\delta>0$ such that if $d(x,y)<\delta$ then $\overline{D}(\left\{
i\geq 0:d(T^{i}x,T^{i}y)\geq \varepsilon \right\}) <\varepsilon$. It was shown in \cite{LTY13} that a t.d.s.
is mean-L-stable if and only if it is mean equicontinuous. Actually the
proof of such equivalence comes from a general argument which can be used to show that every mean condition has an equivalent
density condition. To be precise, one only need to show the following easy observation: Let $M\ge 0,$ and $\left\{ a_{i}\right\} $ be a sequence of reals with $0\leq a_{i}\leq M.$ We have that

$\cdot $ If $\limsup_{n\to\infty} \tfrac{1}{n}\sum\limits_{i=0}^{n-1}a_{i}\leq \delta $
then $\overline{D}(\{ i:a_{i}\geq \sqrt{\delta }\}) \leq \sqrt{%
\delta }$;

$\cdot $ If $\overline{D}(\left\{ i\geq 0:a_{i}\geq \delta \right\}) \leq
\delta $ then $\limsup_{n\to\infty} \tfrac{1}{n}\sum\limits_{i=0}^{n-1}a_{i}\leq
(M+1)\delta$.

\item\label{rem:equi:4}
There exists almost equicontinuous transitive t.d.s. $(X,T)$ which are not mean equicontinuous \cite{LTY13}.
\item\label{rem:equi:5}
Let $(X,T)$ an almost equicontinuous t.d.s. which is not mean equicontinuous. By \cite[Theorem 1.3]{GW93} such system is not an $E$-system. Thus if $A$ is the closure of the union of the supports of all invariant measures then $A$ is a proper closed invariant subset of $X$. Consider the factor map  $\pi\colon X\to Y$ that maps $A$ to a point $p$. Observe that $A$ contains no transitive points and the restricted map $\pi^\prime=\pi|_{X\setminus A}\colon X\setminus A\to Y\setminus \{p\}$ is a homeomorphism. This means that $\pi$ is an almost one-to-one map. Besides, it is easy to see that $Y\times Y$ is uniquely ergodic (i.e., it has only one invariant measure) with the only invariant measure concentrated at the point $(p,p)$. By \cite[Corollary 3.7]{LTY13} the t.d.s. on $Y$ is mean equicontinuous. This shows that every almost equicontinuous transitive t.d.s. is an almost one-to-one extension of a mean equicontinuous one. For a comparison note that factors of almost equicontinuous systems may not even be almost mean equicontinuous (see, e.g., \cite[Remark 4.4]{LTY13}).
\end{enumerate}

\end{rem}

For transitive and minimal systems we can state dichotomies between an
equicontinuous side and a sensitive side. It is known that each transitive
t.d.s. is either $Q$-sensitive (where $Q=$ ordinary or mean or Banach-mean or diam-mean) or almost $Q$-equicontinuous, and that a minimal t.d.s. is either $Q$-sensitive or $Q$-equicontinuous, see \cite{AAB96, LTY13, GM14} respectively. In this paper we will present more characterizations on dichotomy. To make it easier to readers, we briefly sketch the proof of the dichotomy theorem:
If $(X,T)$ is transitive and $U\subset X$ is a nonempty open set, then either
\begin{enumerate}[(i)]
\item $U$ contains a $Q$-equicontinuous point  or
\item there are no $Q$-equicontinuous points in $U$.
\end{enumerate}
If (\text{ii}) holds, then there exists a transitive point $x \in U$ which is not a $Q$-equicontinuous point. It is easy to see that $Tx$ is not a $Q$-equicontinuous point. Using a triangle inequality argument (e.g. see Proposition 5.3 in \cite{LTY13}), one can show that $(X,T)$ is $Q$-sensitive.
If (\text{i}) holds, then all transitive points in $U$ are $Q$-equicontiuous points, and so are all the transitive points in $T^{-i}U$, $i=1,2,3,\dots$.
Since $\bigcup_{i=1}^{\infty}T^{-i}U$ contains all the transitive points, it means that every transitive point is a $Q$-equicontinuous point and then $(X,T)$ is
almost $Q$-equicontinuous.

\begin{rem}\label{rem:dich}
\begin{enumerate}[(1)]
\item\label{rem:dich:1}
The set of transitive points is equal to the set of equicontinuous points of a transitive almost equicontinuous t.d.s. (see, e.g., \cite{AAB96}). This implies that a transitive equicontinuous t.d.s. must be minimal. On the other hand there are transitive non-minimal mean equicontinuous t.d.s. \cite{LTY13}. In the constructed example the set of its non-transitive mean equicontinous points forms a dense subset.
The Remark  \ref{rem:equi}\eqref{rem:equi:4} above provides another example where some, but not all, of the non-transitive points are mean equicontinuous points (just note that $Y$ is mean equicontinuous and the restricted factor map $\pi|_{X\setminus \pi^{-1}(U)}\colon X\setminus \pi^{-1}(U)\to Y\setminus U$ is a uniform isomorphism for any open set $U$ which contains $p$).

\item\label{rem:dich:2}
There are some easy conditions that imply a t.d.s. is not mean equicontinuous. Since every transitive mean equicontinuous t.d.s. is uniquely ergodic (see, e.g., \cite{LTY13}) we obtain that the following families are not mean equicontinuous:

$\cdot $ transitive systems with more than one minimal subsystem, e.g., non-minimal M-system;

$\cdot $ non-minimal E-systems (since it has at least two invariant measure: the one with full support and the one concentrated on a minimal proper subset).
\end{enumerate}
\end{rem}

\medskip \textbf{$\bullet$ Symbolic dynamics} \medskip

Let $\Sigma _{2}^{+}=\{0,1\}^{{\mathbb{N}}}$ endowed with the Cantor product
topology (given by the discrete topology on $\{0,1\}$). A compatible metric
on $\Sigma _{2}^{+}$ is defined by $d(x,y)=0$ if $x=y$ otherwise $d(x,y)={1}/%
{i}$ with $i=\min \{j\in {\mathbb{N}}\colon x_{j}\neq y_{j}\}$. We have that
$\Sigma _{2}^{+}$ is compact, and the \textit{shift map} $\sigma \colon
\Sigma _{2}^{+}\rightarrow \Sigma _{2}^{+},$ defined by $\sigma
(x)_{n}=x_{n+1}$ for $n\in \mathbb{N},$ is continuous. We often refer $%
(\Sigma _{2}^{+},\sigma )$ as the \textit{full shift}, and any compact $%
\sigma $-invariant subsystem $X\subset \Sigma _{2}^{+}$, as a \emph{subshift}%
.

Fix $n\in \mathbb{N}$, we call $w\in \{0,1\}^{n}$ a \textit{word of length $%
n $} and write $|w|=n$. Let $\mathfrak{o}(w)=\#\{i\in \mathbb{N}\colon
w_{i}=1\}$ be the number of occurrences of symbol $1$ in $w$. For any two
words $u=u_{1}u_{2}\dots u_{n}$ and $v=v_{1}v_{2}\dots v_{m}$, we define the
\textit{concatenation} of $u,v$ by $uv=u_{1}u_{2}\dots u_{n}v_{1}v_{2}\dots
v_{m}$ or equivalently by $u\sqcup v$. By the same manner we define by $%
u^{m} $ the concatenation of $m$ copies of $u$ for some $m\in \mathbb{N}$,
and $u^{\infty }$ the infinite concatenation of $u$. Let $X$ be a subshift
of $\Sigma _{2}^{+}$ and $x=x_{1}x_{2}\dots \in X$, for any $t,j\in {\mathbb{%
Z}}_{+}$, we denote $x_{[t,t+j]}=x_{t}x_{t+1}\dots x_{t+j}$. We say that a
word $w=w_{1}w_{2}\dots w_{n}$ appears in $x$ at position $t$ if $%
x_{t+j-1}=w_{j}$ for $j=1,2,\dots ,n$ i.e. $x_{[t,t+n-1]}=w$. By ${\mathcal{L%
}}(X)$ we mean the \textit{language} of subshift $X$, which is the set
consisting of all words that can appear in some $x\in X$, and by $\mathcal{L}%
_{n}(X)$ the set of all words of length $n$ in ${\mathcal{L}}(X)$. For any
word $u\in \mathcal{L}_{n}(X)$ its \emph{cylinder set} is defined by $%
[u]=\{x\in X\colon x_{1}x_{2}\dots x_{n}=u\}$. Note that the cylinder sets $%
\{[u]\colon u\in {\mathcal{L}}(X)\}$ form a basis of the topology of $X$. A
subshift is transitive if and only if it contains a transitive point. %

\part{Main results}

\section{Proof of Theorem \protect\ref{thm:Diam-Mean-Sen not Mean-Sen}}

\label{sect:Proof-Of-Theorem 1.3}

In this section we will prove Theorem \ref{thm:Diam-Mean-Sen not Mean-Sen},
i.e., we construct a subshift which is cofinitely sensitive but not Banach
mean sensitive. The needed notations was summarized in Section \ref%
{sect:preliminaries}.

We are working with the non-common notation of using $A\sqcup B=AB$ to
denote the concatenation of two words and $\sqcup _{i=a}^{b}A^{i}$ the
concatenation of the words $\left\{ A^{i}\right\} _{i=a}^{b}$ . By
convention, if $a>b$ we assume $\sqcup _{i=a}^{b}A^{i}$ is an empty word.

To begin with, let $A_{1}=111$ and $B_{1}=000$. For $n\geq 2$ we recursively
write
\begin{equation*}
A_{n}=A_{n-1}0^{k_{n-1}}B_{n-1}0^{k_{n-1}}A_{n-1}
\end{equation*}%
and
\begin{align*}
B_{n}& =\sqcup _{i=1}^{|A_{n}|}A_{n-1}0^{i-1}10^{|A_{n}|-i}\sqcup
A_{n-1}0^{|A_{n}|} \\
& =A_{n-1}10^{|A_{n}|-1}\dots A_{n-1}0^{i-1}10^{|A_{n}|-i}\dots
A_{n-1}0^{|A_{n}|-1}1A_{n-1}0^{|A_{n}|}
\end{align*}%
where the sequence $\{k_{n}\}$ is required to satisfy that $k_{n}\geq
n(2|A_{n}|+|B_{n}|)$.

Let $x=\lim_{n\rightarrow \infty }A_{n}0^{\infty }$ and $X=\overline{\text{%
orb}(x,\sigma )}$. By construction $x$ is clearly a recurrent point and
consequently $X$ is a transitive t.d.s. Before proceeding, we describe
several properties of $X$ with the following lemmas.

\begin{lem}
\label{lem:Density-1} For any $m > n$ and $\omega\in\mathcal{L}%
_{t}(A_{m})\cup\mathcal{L}_{t}(B_{m})$ with $t
(=t_{n})=|A_{n}|+2k_{n}+|B_{n}|$, we have that
\begin{equation*}
\mathfrak{o}(\omega)=\#\{i\in\mathbb{N}\colon \omega_{i}=1\}
\le|A_{n}|+|B_{n}|.
\end{equation*}
\end{lem}

\begin{proof}
We will prove this using induction on $m$.

Assume that $m=n+1$. If $\omega \in \mathcal{L}_{t}(A_{n+1})$, we clearly
have $\mathfrak{o}(\omega )<|A_{n}|+|B_{n}|$. That is because $\omega $ can only be
one of the following forms:
\begin{equation*}
A_{n}0^{k_{n}}B_{n}0^{k_{n}}\ ,\ 0^{k_{n}}B_{n}0^{k_{n}}A_{n}
\end{equation*}%
and
\begin{equation*}
{(A_{n})}_{[|A_{n}|-l,|A_{n}|]}0^{k_{n}}B_{n}0^{k_{n}}{(A_{n})}%
_{[1,|A_{n}|-l-1]},0\leq l<|A_{n}|-1.
\end{equation*}


If $\omega \in \mathcal{L}_{t}(B_{n+1})$, then
\begin{align*}
B_{n+1}& =\sqcup _{i=1}^{|A_{n+1}|}A_{n}0^{i-1}10^{|A_{n+1}|-i}\sqcup
A_{n}0^{|A_{n+1}|} \\
& =A_{n}10^{|A_{n+1}|-1}\dots A_{n}0^{i-1}10^{|A_{n+1}|-i}\dots
A_{n}0^{|A_{n+1}|-1}1A_{n}0^{|A_{n+1}|}.
\end{align*}%
With this we obtain that $\mathfrak{o}(\omega )\leq
|A_{n}|+1<|A_{n}|+|B_{n}|.$

Assume that for any $n<l\leq m-1$ and $\omega \in \mathcal{L}_{t}(A_{l})\cup
\mathcal{L}_{t}(B_{l})$. We conclude that
\begin{equation*}
\mathfrak{o}(\omega )\leq |A_{n}|+|B_{n}|.
\end{equation*}%
If $\omega \in \mathcal{L}_{t}(A_{m}),$ we have that
\begin{align*}
\omega \in & \mathcal{L}_{t}(A_{m-1})\cup \mathcal{L}%
_{t}(A_{m-1}0^{k_{m-1}})\cup \mathcal{L}_{t}(0^{k_{m-1}}A_{m-1}) \\
& \cup \mathcal{L}_{t}(B_{m-1})\cup \mathcal{L}_{t}(B_{m-1}0^{k_{m-1}})\cup
\mathcal{L}_{t}(0^{k_{m-1}}B_{m-1}).
\end{align*}%
Thus we also conclude $\mathfrak{o}(\omega )\leq |A_{n}|+|B_{n}|$.

Finally, assume $\omega \in \mathcal{L}_{t}(B_{m}).$ We express
\begin{equation*}
B_{m}=\sqcup _{i=1}^{|A_{m}|}A_{m-1}0^{i-1}10^{|A_{m}|-i}\sqcup
A_{m-1}0^{|A_{m}|}.
\end{equation*}

Since each $A_{l}$ ($l>n$) starts and ends with $A_{n}$ we have that
\begin{align*}
\omega \in & \mathcal{L}_{t}(A_{m-1})\cup \bigcup_{i=1}^{|A_{m}|}\mathcal{L}%
_{t}(0^{i-1}10^{|A_{m}|-i})\cup \mathcal{L}_{t}(0^{|A_{m}|}) \\
& \cup \bigcup_{i=1}^{|A_{m}|}\bigcup_{p=1}^{t-1}\mathcal{L}%
_{t}(0^{i-1}10^{|A_{m}|-i}({A_{n+1}})_{[1,p]}) \\
& \cup \bigcup_{i=1}^{|A_{m}|}\bigcup_{q=0}^{t-2}\mathcal{L}_{t}(({A_{n+1}}%
){}_{[|A_{n+1}|-q,|A_{n+1}|]}0^{i-1}10^{|A_{m}|-i}) \\
& \cup \bigcup_{q=0}^{t-2}\mathcal{L}_{t}(({A_{n+1}}%
){}_{[|A_{n+1}|-q,|A_{n+1}|]}0^{|A_{m}|}).
\end{align*}%
Once again we conclude $\mathfrak{o}(\omega )\leq |A_{n}|+|B_{n}|$.
\end{proof}

Let $y\in X$. We define $E_{y}:=\{i\in {\mathbb{Z}}_{+}\colon y_{i+1}=1\}$.

\begin{lem}
\label{lem:Density-2} For any $y\in X$ we have that $BD^{\ast }(E_{y})=0$.
\end{lem}

\begin{proof}
We will prove this result in two cases.

Assume $y\in \orb(x,\sigma ).$ Then $y=\sigma ^{a}x$ is a transitive point
for some $a\in {\mathbb{Z}}_{+}$. Note that for each $%
n>t_{1}=|A_{1}|+2k_{1}+|B_{1}|$ there is a unique $m(=m(n))\in \mathbb{N}$
such that $t_{m}\leq n=r_{m}t_{m}+s_{m}<t_{m+1}$ with $%
t_{m}=|A_{m}|+2k_{m}+|B_{m}|$, $r_{m}\in \mathbb{N}$ and $0\leq s_{m}\leq
t_{m}-1$. Thus by Lemma \ref{lem:Density-1} we have that for any interval $%
[M,N-1]$ with length $n$
\begin{align*}
\frac{\#\{E_{y}\cap \lbrack M,N-1]\}}{n}& =\frac{\mathfrak{o}(y_{[M+1,M+n]})%
}{n} \\
& \leq \frac{\mathfrak{o}(y_{[M+1,M+(r_{m}+1)t_{m}]})}{r_{m}t_{m}} \\
& \leq \frac{(r_{m}+1)(|A_{m}|+|B_{m}|)}{r_{m}(|A_{m}|+2k_{m}+|B_{m}|)} \\
& \leq \frac{(r_{m}+1)(|A_{m}|+|B_{m}|)}{%
r_{m}[|A_{m}|+2m(2|A_{m}|+|B_{m}|)+|B_{m}|]}.
\end{align*}%
Considering this and the fact that $m\rightarrow \infty $ when $n\rightarrow
\infty $, we conclude
\begin{align*}
BD^{\ast }(E_{y})& =\limsup_{n\rightarrow \infty }\left\{ \sup_{N-M=n}\frac{%
\#\{E_{y}\cap \lbrack M,N-1]\}}{n}\right\} \\
& \leq \limsup_{m\rightarrow \infty }\frac{(r_{m}+1)(|A_{m}|+|B_{m}|)}{%
r_{m}[|A_{m}|+2m(2|A_{m}|+|B_{m}|)+|B_{m}|]}=0.
\end{align*}

Now assume $y\in X\setminus \orb(x,\sigma ).$ There exists a sequence $%
\{a_{i}\}_{i=1}^{\infty }\subset {\mathbb{Z}}_{+}$ such that $\sigma
^{a_{i}}x\rightarrow y$. Hence for each word $y_{[M+1,M+n]}$ with $M\in {%
\mathbb{Z}}_{+},n\in \mathbb{N}$, there is an $a_{b_{n}}\in \mathbb{N}$ such
that
\begin{equation*}
y_{[1,M+n]}=\sigma ^{a_{b_{n}}}x_{[1,M+n]}=x_{[a_{b_{n}}+1,a_{b_{n}}+M+n]}.
\end{equation*}%
Using Lemma \ref{lem:Density-1} (and a similar argument as before) we
conlude
\begin{align*}
BD^{\ast }(E_{y})& =\limsup_{n\rightarrow \infty }\left\{ \sup_{N-M=n}\frac{%
\mathfrak{o}(x_{[a_{b_{n}}+M+1,a_{b_{n}}+M+n]})}{n}\right\} \\
& \leq \limsup_{m\rightarrow \infty }\frac{(r_{m}+1)(|A_{m}|+|B_{m}|)}{%
r_{m}[|A_{m}|+2m(2|A_{m}|+|B_{m}|)+|B_{m}|]}=0.
\end{align*}
\end{proof}

Now we are ready to prove Theorem \ref{thm:Diam-Mean-Sen not Mean-Sen}.

\begin{proof}[proof of Theorem \protect\ref{thm:Diam-Mean-Sen not Mean-Sen}]

Let $X$ be the subshift defined above. First we will prove $X$ is cofinitely
sensitive. Let $U\subset X$ be a non-empty open set. Since $x\in X$ is a
transitive point, there exists $m\in \mathbb{N}$ such that $\sigma ^{m}x\in
U $. Furthermore there exists $s\in \mathbb{N}$ such that the cylinder $%
[x_{[m+1,m+s]}]\subset U$. By the construction we can pick an $i\in \mathbb{N%
}$ such that $x_{[m+1,m+s]}$ is a subword of $A_{i}$. Without loss of
generality, we can further assume the word $A_{i}$ ends with $x_{[m+1,m+s]}$
(choose large enough $s$ if necessary). Also, observe that each word $%
A_{i_{1}}$ ends with $A_{i_{2}}$ and appears in $B_{i_{3}}$ for any $%
i_{2}<i_{1}<i_{3}\in \mathbb{N}$. Thus it is not hard to see that the
sequences $z_{j}=x_{[m+1,m+s]}0^{j}10^{\infty }$ ($j=0,1,2,\ldots $) are
eventually the points of $X$, and furthermore $\{z_{j}\colon j\in {\mathbb{Z}%
}_{+}\}\subset \lbrack x_{[m+1,m+s]}]\subset U$. Consequently we have
\begin{equation*}
N_{T}(U,1/2)=\{n\in {\mathbb{Z}}_{+}\colon \diam(\sigma ^{n}U)>1/2\}\supset {%
\mathbb{Z}}_{+}\setminus \lbrack 0,m+s],
\end{equation*}%
and the cofinite sensitivity follows.

Let $\varepsilon >0$ and $\eta \in (0,\varepsilon )$. For $\eta $ there is a
$K\in \mathbb{N}$ such that $1/{(K+1)}\leq \eta <1/K$. For any pair $%
(y_{1},y_{2})\in X\times X$ we can use Lemma \ref{lem:Density-2} to obtain $%
BD^{\ast }(E_{y_{1}}\cup E_{y_{2}})=0$. This implies that if $F=\{i\in {%
\mathbb{Z}}_{+}\colon d(\sigma ^{i}y_{1},\sigma ^{i}y_{2})\geq \eta \}$ then
\begin{equation*}
BD^{\ast }(F)\leq \limsup_{N-M\rightarrow \infty }\frac{2K\#\{(E_{y_{1}}\cup
E_{y_{2}})\cap \lbrack M,N-1]\}}{N-M}=0.
\end{equation*}%
Hence similarly to Remark \ref{rem:equi}\eqref{rem:equi:3},
\begin{align*}
\limsup_{N-M\rightarrow \infty }& \frac{1}{N-M}\sum_{i=M}^{N-1}d(\sigma
^{i}y_{1},\sigma ^{i}y_{2}) \\
& \leq \limsup_{N-M\rightarrow \infty }\frac{1}{N-M}(\#\{F\cap \lbrack
M,N-1]\}\cdot \diam(X)+\eta \cdot (N-M)) \\
& \leq BD^{\ast }(F)\cdot \diam(X)+\eta <\varepsilon .
\end{align*}%
Since the choice of pair $(y_{1},y_{2})$ is arbitrary, we have that $X$ is
Banach-mean equicontinuous, thus not Banach-mean sensitive.
\end{proof}

\section{Proof of Theorem \protect\ref{thm:Dich-E-system}}

\label{sect:Proof-Of-Theorem 1.4}

In this section we show that positive topological entropy and Devaney chaos
does not imply mean sensitivity. As a corollary we show that mean
sensitivity and mean equicontinuity have no dichotomy for $E$-systems. It is
fair to mention that the following construction is inspired by the techniques
in \cite{LTY13} and \cite{LOYZ15}, nonetheless, the proof for this result
has several extra technical issues.

We start with an arbitrarily binary minimal subshift $(Y,\sigma )$. Pick a
point $y=y_{1}y_{2}\ldots \in Y$ and denote the word $C_{n}=y_{1}\ldots
y_{n} $ for each $n\in \mathbb{N}$. Set $A_{1}=101$, $B_{1}=C_{1}$. Now we
recursively define
\begin{equation*}
A_{n+1}=A_{n}0^{k_{n}}B_{n}0^{k_{n}}A_{n}
\end{equation*}%
and
\begin{align*}
B_{n+1}=& C_{n+1}\sqcup _{i=1}^{n}(A_{i}0^{|A_{i+1}|})^{n+1-i}\text{ \ \ } \\
=& C_{n+1}(A_{1}0^{|A_{2}|})^{n}(A_{2}0^{|A_{3}|})^{n-1}\ldots
(A_{n}0^{|A_{n+1}|})^{1},
\end{align*}%
where the sequence $\{k_{n}\}$ satisfies the following properties:

\begin{enumerate}
\item \label{cond:1} $\frac{k_{m}}{|A_{n}|+2k_{n}+|B_{n}|}>\frac{|B_{m}|}{%
|B_{n}|}$ for all $1\leq n< m$ and

\item \label{cond:2} $k_{n}\ge n(2|A_{n}|+|B_{n}|)$.
\end{enumerate}

Let $x=\lim_{n\rightarrow \infty }A_{n}0^{\infty }$ and $X=\overline{\orb%
(x,\sigma )}$. From the construction it is clear that $(Y,\sigma )$ is a
minimal subshift of $(X,\sigma )$, and the system $(X,\sigma )$ is
transitive and has dense periodic points (of the form $\sigma
^{t}(A_{n}0^{|A_{n+1}|})^{\infty })$. From this we obtain the following
result.

\begin{prop}
\label{prop:Devaney} $(X,\sigma)$ is a Devaney chaotic t.d.s.
\end{prop}

We will later prove that the subshift $(X,\sigma )$ satisfies the following
property:

\begin{prop}
\label{prop:P-system} $(X,\sigma)$ is almost mean equicontinuous.
\end{prop}

From the construction and Propositions ~\ref{prop:Devaney} and~\ref{prop:P-system}, we have that:

\begin{prop}
\label{prop:P-system-1} In the full shift $(\Sigma _{2}^{+},\sigma )$, every
minimal subshift $(Y,\sigma )$ is contained in a Devaney chaotic subshift $%
(X,\sigma )$ which is almost mean equicontinuous.
\end{prop}

Note that since the system in Proposition \ref{prop:P-system-1} is transitive but not uniquely ergodic,
by Remark \ref{rem:dich}\eqref{rem:dich:2} it can not be mean equicontinuous.

\medskip
Now we can prove Theorem \ref{thm:Dich-E-system}.

\begin{proof}[Proof of Theorem \protect\ref{thm:Dich-E-system}]
Let $(Y,\sigma)$ be a minimal subshift with positive topological entropy. Using the fixed $Y$ we
construct a subshift $(X,\sigma)$ as in Proposition \ref{prop:P-system-1}. Then $(X,\sigma)$ is Devaney chaotic and almost mean equicontinuous, and so it is not
mean sensitive \cite{LTY13, GM14}. Besides,  $(X,\sigma)$ also has positive topological entropy because it contains $(Y,\sigma)$ as a subsystem. This implies that   
Theorem \ref{thm:Dich-E-system} holds, completing the proof.
\end{proof}

Consequently we have the proof of Corollary \ref{Dich-E-Cor}.

\begin{proof}[Proof of Corollary \protect\ref{Dich-E-Cor}]
Note that every Devaney chaotic t.d.s. is an $E$-system and mean
equicontinuous systems have zero entropy (\cite{LTY13, GM14}). Then
combining with Theorem \ref{thm:Dich-E-system} the corollary follows. One
can also yield the corollary from Proposition~\ref{prop:P-system-1}.
\end{proof}

So it remains to prove Proposition ~\ref{prop:P-system}. For this we need
more information about the structure of $(X,\sigma )$, which will be
obtained in the following series of lemmas. To help the readers to get a
good understanding of these lemmas we will describe the relations between
them: Lemma \ref{lem:Count-3} is exactly what we need in the proof of
Proposition ~\ref{prop:P-system}; Lemma \ref{lem:P-system-almostMeanEqui-1},
Lemma \ref{lem:P-system-almostMeanEqui-2} and Lemma \ref%
{lem:P-system-almostMeanEqui-3} are served for proving Lemma \ref%
{lem:Count-3}.

\medskip
Observe that each $A_{j}$ starts and ends with $A_{i}$ for all $%
1\leq i<j\in\mathbb{N}$. Let $m\geq n+1$. Using that $%
A_{n+1}=A_{n}0^{k_{n}}B_{n}0^{k_{n}}A_{n}$ one can check that there exists
an $M\in\mathbb{N}$ and a sequence $\left\{ m_{i}\right\} _{i=1}^{M}$ with $%
n\le m_{i}\leq m-1$ such that
\begin{align}
A_{m} & =\sqcup_{i=1}^{M}A_{n}0^{k_{m_{i}}}B_{m_{i}}0^{k_{m_{i}}}\sqcup A_{n}
\label{express-1} \\
& =A_{n}0^{k_{n}}B_{n}0^{k_{n}}\ A_{n}0^{k_{n+1}}B_{n+1}0^{k_{n+1}}\ldots
A_{n}0^{k_{n}}B_{n}0^{k_{n}}A_{n} .  \notag
\end{align}

\begin{lem}
\label{lem:Count-1}\label{lem:P-system-almostMeanEqui-1} For each $m\geq n$
and $1\leq s\leq|A_{n}|+2k_{m}+|B_{m}|$, we have that
\begin{equation*}
\mathfrak{o}((A_{n}0^{k_{m}}B_{m}0^{k_{m}}){}_{[1,s]})\leq\max\Bigl\{1,%
\frac {s}{|A_{n}|+2k_{n}+|B_{n}|}\Bigr\}\cdot(|A_{n}|+|B_{n}|).
\end{equation*}
\end{lem}

\begin{proof}
It is clearly true for $m=n$. Assume $m>n$. The result is easy to check for $%
s\in \lbrack 1,|A_{n}|+k_{m}]$. Now, since the sequence $\{k_{n}\}$
satisfies condition \eqref{cond:1} we have that $k_{m}>2k_{n}$ and $|B_{m}|<%
\frac{|B_{n}|\cdot k_{m}}{|A_{n}|+2k_{n}+|B_{n}|}$. This leads to
\begin{align*}
\mathfrak{o}(A_{n}0^{k_{m}}B_{m}0^{k_{m}}{}_{[1,s]})& \leq
|A_{n}|+|B_{m}|\leq |A_{n}|+\frac{|B_{n}|\cdot k_{m}}{|A_{n}|+2k_{n}+|B_{n}|}
\\
& \leq \frac{(|A_{n}|+k_{m})(|A_{n}|+|B_{n}|)}{|A_{n}|+2k_{n}+|B_{n}|} \\
& \leq \frac{s}{|A_{n}|+2k_{n}+|B_{n}|}(|A_{n}|+|B_{n}|)
\end{align*}%
for all $s>|A_{n}|+k_{m}$, completing the proof.
\end{proof}


\begin{lem}
\label{lem:P-system-almostMeanEqui-2} For each $m\geq n+1$, we have

\begin{enumerate}
\item \label{lem:P-system-almostMeanEqui-2:1} $\mathfrak{o}(A_{m} )\leq (%
\frac{|A_{m}|}{|A_{n}|+2k_{n}+|B_{n}|}+1)\cdot(|A_{n}|+|B_{n}|)$;

\item \label{lem:P-system-almostMeanEqui-2:2} $\mathfrak{o}%
(A_{m}0^{|A_{m+m^{\prime}}|} )=\mathfrak{o}(A_{m}) \leq(\frac{%
|A_{m}|+|A_{m+m^{\prime}}|}{|A_{n}|+2k_{n}+|B_{n}|})\cdot(|A_{n}|+|B_{n}|)$
for any $m^{\prime}\ge1$;

\item \label{lem:P-system-almostMeanEqui-2:3} $\mathfrak{o}%
(A_{n}0^{|A_{m}|}{}_{[1,s]} )\leq|A_{n}|+|B_{n}|$ for any $1 \le s
\le|A_{n}|+|A_{m}|$. 

\end{enumerate}
\end{lem}

\begin{proof}
\eqref{lem:P-system-almostMeanEqui-2:3} is trivial and obviously $%
\eqref{lem:P-system-almostMeanEqui-2:1}\Longrightarrow %
\eqref{lem:P-system-almostMeanEqui-2:2}$. So we only need to show %
\eqref{lem:P-system-almostMeanEqui-2:1}.

Let $M\in\mathbb{N}$ and $\left\{ m_{i}\right\} _{i=1}^{M}$ be such that $%
n\le m_{i}\leq m-1$ and
\begin{equation*}
A_{m}=\sqcup_{i=1}^{M}A_{n}0^{k_{m_{i}}}B_{m_{i}}0^{k_{m_{i}}}\sqcup A_{n}.
\end{equation*}
Observe that for each $r\geq n$,
\begin{equation*}
|A_{n}0^{k_{r}}B_{r}0^{k_{r}}|=|A_{n}|+2k_{r}+|B_{r}|%
\geq|A_{n}|+2k_{n}+|B_{n}|,
\end{equation*}
By Lemma \ref{lem:P-system-almostMeanEqui-1},
\begin{equation*}
\mathfrak{o}(A_{n}0^{k_{r}}B_{r}0^{k_{r}})\leq\frac{%
|A_{n}0^{k_{r}}B_{r}0^{k_{r}}|}{|A_{n}|+2k_{n}+|B_{n}|}%
\cdot(|A_{n}|+|B_{n}|),
\end{equation*}
so
\begin{align*}
\mathfrak{o}(A_{m}) & =\sum_{i=1}^M\mathfrak{o}%
(A_{n}0^{k_{m_i}}B_{m_i}0^{k_{m_i}})+\mathfrak{o}(A_{n}) \\
& \leq\sum_{i=1}^M\frac{|A_{n}0^{k_{m_i}}B_{m_i}0^{k_{m_i}}|}{%
|A_{n}|+2k_{n}+|B_{n}|}\cdot(|A_{n}|+|B_{n}|)+(|A_{n}|+|B_{n}|) \\
& \leq(\frac{|A_{m}|}{|A_{n}|+2k_{n}+|B_{n}|}+1)\cdot(|A_{n}|+|B_{n}|),
\end{align*}
completing the proof.
\end{proof}

Since $y$ is a minimal point, the word $0$ or $1$ appears in $y$
syndetically. This implies that there exists $N\in\mathbb{N}$ such that $%
A_{N}$ does not appear in $y$. We will use this $N$ in the following lemma.

\begin{lem}
\label{lem:P-system-almostMeanEqui-3} Let $n\geq N$, $m\geq n+1$ and $1\leq
t\leq|B_{m}|$ such that $(B_{m}){}_{[t,t+|A_{n}|-1]}=A_{n}$. Then for each $%
0\leq s\leq|B_{m}|+k_{m}-t$ we have
\begin{equation*}
\mathfrak{o}((B_{m}0^{k_{m}}){}_{[t,t+s]})\leq\Bigl(\frac{s+1}{%
|A_{n}|+2k_{n}+|B_{n}|}+1\Bigr)\cdot(|A_{n}|+|B_{n}|).
\end{equation*}
\end{lem}

\begin{proof}
By definition we have that
\begin{equation*}
B_{m}=C_{m}\sqcup_{i=1}^{m-1}(A_{i}0^{|A_{i+1}|})^{m-i}\text{ .\ }
\end{equation*}
Considering the choice of $N$, $A_{n}$ does not appear as a subword of $%
C_{m}\sqcup_{i=1}^{n-1}(A_{i}0^{|A_{i+1}|})^{m-i}$, nonetheless for each $%
n\leq i\leq m-1$, $A_{n}$ is contained in $A_{i}0^{|A_{i+1}|}$.

By hypothesis we have that $(B_{m}){}_{[t,t+|A_{n}|-1]}=A_{n}$. There exist $%
n\leq r\leq m-1$ and $0\le l<m-r$ such that $t$ locates in some $A_{r}$ of $%
B_{m}$, i.e.,
\begin{equation*}
\left\vert C_{m}\sqcup_{i=1}^{r-1}(A_{i}0^{|A_{i+1}|})^{m-i}\sqcup
(A_{r}0^{|A_{r+1}|})^{l}\right\vert <t\leq\left\vert
C_{m}\sqcup_{i=1}^{r-1}(A_{i}0^{|A_{i+1}|})^{m-i}%
\sqcup(A_{r}0^{|A_{r+1}|})^{l+1}\right\vert .
\end{equation*}
While there exists $M\in\mathbb{N}$ and a sequence $\left\{ m_{i}\right\}
_{i=1}^{M}$ with $n\le m_{i}\le r-1$ such that
\begin{equation*}
A_{r}=\sqcup_{i=1}^{M}A_{n}0^{k_{m_{i}}}B_{m_{i}}0^{k_{m_{i}}}\sqcup A_{n}.
\label{A_r-expression2}
\end{equation*}
We assume there exists $0\leq a\leq M$ such that
\begin{equation*}
t=\left\vert C_{m}\sqcup_{i=1}^{r-1}(A_{i}0^{|A_{i+1}|})^{m-i}\sqcup
(A_{r}0^{|A_{r+1}|})^{l}\right\vert +\left\vert
\sqcup_{i=1}^{a}A_{n}0^{k_{m_{i}}}B_{m_{i}}0^{k_{m_{i}}}\right\vert +1.
\end{equation*}
This may not be the case, as $A_{n}$ appears also as subwords of some $%
B_{m_{i}}.$ We will explain later what to do if this is not the case.

This implies that%
\begin{align}  \label{express-Bm}
(B_{m}0^{k_{m}}){}_{[t,t+s]}=& (\sqcup
_{i=a+1}^{M}A_{n}0^{k_{m_{i}}}B_{m_{i}}0^{k_{m_{i}}}\sqcup A_{n}0^{|A_{r+1}|}
\\
& \sqcup (A_{r}0^{|A_{r+1}|})^{m-r-l-1}\sqcup
_{i=r+1}^{m-1}(A_{i}0^{|A_{i+1}|})^{m-i}\sqcup 0^{k_{m}})_{[0,s]}.  \notag
\end{align}%
Depending on the location of $t+s$, $(B_{m}0^{k_{m}}){}_{[t,t+s]}$ can be
rewritten as the following concatenations:
\begin{equation*}
(B_{m}0^{k_{m}}){}_{[t,t+s]}=\sqcup _{j=1}^{9}E_{j},
\end{equation*}%
where for each $\ell =1,\ldots ,9$, $E_{\ell }$ is either empty, or
satisfies the following:

\begin{enumerate}
\item \label{possible-case:1} $E_{1}=%
\sqcup_{i=a+1}^{M}A_{n}0^{k_{m_{i}}}B_{m_{i}}0^{k_{m_{i}}}$;

\item \label{possible-case:2} $E_{2}=A_{n} 0^{|A_{r+1}|}$;

\item \label{possible-case:3} $E_{3}=(A_{r}0^{|A_{r+1}|})^{u}$ for some $%
0\le u\le m-r-l-1$;

\item \label{possible-case:4} $E_{4}=
\sqcup_{i=r+1}^{v-1}(A_{i}0^{|A_{i+1}|})^{m-i}$ for some $r+2\le v\le m-1$;

\item \label{possible-case:5} $E_{5}=(A_{v}0^{|A_{v+1}|})^{w}$ for some $%
0\le w\le m-v$;

\item \label{possible-case:6} $E_{6}=%
\sqcup_{i=1}^{p}A_{n}0^{k_{b_{i}}}B_{b_{i}}0^{k_{b_{i}}}$ for some $p\in%
\mathbb{N}$ and $\{b_{i}\}_{i=1}^{p}$ with $n\le b_{i}\le m-2$;

\item \label{possible-case:7} $%
E_{7}=A_{n}0^{k_{c}}B_{k_{c}}0^{k_{c}}{}_{[1,e_{1}]}$ for some $c \ge n$ and
$1\le e_{1}\le|A_{n}|+|B_{k_{c}}|+2k_{c}$;

\item \label{possible-case:8} $E_{8}=A_{n}0^{|A_{d}|}{}_{[1,e_{2}]}$ for
some $d \ge n$ and $1 \le e_{2} \le|A_{n}|+|A_{d}|$;

\item \label{possible-case:9} $E_{9}=A_{n}0^{|A_{m}|}0^{k_{m}}{}_{[1,e_{3}]}$
for some $1\le e_{3}\le|A_{m}|+|A_{n}|+k_{m}$.
\end{enumerate}

Now we make some comments about the decomposition for better understanding.
Observing the expression of \eqref{express-Bm}, if $t+s$ locates in some $%
A_{t}0^{|A_{t+1}|}\ (n\leq t\leq m-1)$ then $\sqcup _{i=1}^{5}E_{i}$ is the
left part determined by the chosen $A_{t}0^{|A_{t+1}|}$. Since $%
A_{t}0^{|A_{t+1}|}=\sqcup
_{i=1}^{q}A_{n}0^{k_{m_{i}}}B_{m_{i}}0^{k_{m_{i}}}\sqcup A_{n}0^{|A_{t+1}|}$
for some $q\in {\mathbb{Z}}_{+}$ and a sequence $\{m_{i}\}_{i=1}^{q}$ with $%
n\leq m_{i}\leq t-1$, $\sqcup _{i=6}^{9}E_{i}$ is the part from the
beginning of the chosen $A_{t}0^{|A_{t+1}|}$ to the location of $t+s$. $%
E_{7},E_{8},E_{9}$ are the possible cases that are left over from the
division and we have that one and only one of them is non-empty. The idea of
the decomposition is to expand $B_{m}0^{k_{m}}$ into the concatenations of
the forms $A_{n}0^{k_{m_{i}}}B_{m_{i}}0^{k_{m_{i}}}$ and $A_{n}0^{|A_{j}|}$.

Once $B_{m}0^{k_{m}}$ is broken as the union of $E_{i}$ as above we can make
use of Lemmas \ref{lem:P-system-almostMeanEqui-1} and \ref%
{lem:P-system-almostMeanEqui-2} to show the following.

\begin{enumerate}
\item $\mathfrak{o}(E_{1})=\sum \mathfrak{o}%
(A_{n}0^{k_{m_{i}}}B_{m_{i}}0^{k_{m_{i}}})\leq \frac{|E_{1}|}{%
|A_{n}|+2k_{n}+|B_{n}|}\cdot (|A_{n}|+|B_{n}|)$;

\item By Lemma \ref{lem:P-system-almostMeanEqui-2}%
\eqref{lem:P-system-almostMeanEqui-2:1}, $\mathfrak{o}(E_{2})=\mathfrak{o}%
(A_{n})\leq \frac{|E_{2}|}{|A_{n}|+2k_{n}+|B_{n}|}\cdot (|A_{n}|+|B_{n}|)$;

\item By Lemma \ref{lem:P-system-almostMeanEqui-2}%
\eqref{lem:P-system-almostMeanEqui-2:2}, $\mathfrak{o}(E_{3})\leq \frac{%
|E_{3}|}{|A_{n}|+2k_{n}+|B_{n}|}\cdot (|A_{n}|+|B_{n}|)$;

\item By Lemma \ref{lem:P-system-almostMeanEqui-2}%
\eqref{lem:P-system-almostMeanEqui-2:2}, $\mathfrak{o}(E_{4})\leq \frac{%
|E_{4}|}{|A_{n}|+2k_{n}+|B_{n}|}\cdot (|A_{n}|+|B_{n}|)$;

\item By Lemma \ref{lem:P-system-almostMeanEqui-2}%
\eqref{lem:P-system-almostMeanEqui-2:2}, $\mathfrak{o}(E_{5})\leq \frac{%
|E_{5}|}{|A_{n}|+2k_{n}+|B_{n}|}\cdot (|A_{n}|+|B_{n}|)$;


\item $\mathfrak{o}(E_{6})=\sum \mathfrak{o}%
(A_{n}0^{k_{b_{i}}}B_{b_{i}}0^{k_{b_{i}}})\leq \frac{|E_{6}|}{%
|A_{n}|+2k_{n}+|B_{n}|}\cdot (|A_{n}|+|B_{n}|)$;

\item By Lemma \ref{lem:P-system-almostMeanEqui-1}, $\mathfrak{o}(E_{7})\leq
\max \Bigl\{1,\frac{|E_{7}|}{|A_{n}|+2k_{n}+|B_{n}|}\Bigr\}\cdot
(|A_{n}|+|B_{n}|)$;

\item $\mathfrak{o}(E_{8})\leq \mathfrak{o}(A_{n})\leq |A_{n}|+|B_{n}|$;

\item $\mathfrak{o}(E_{9})\leq \mathfrak{o}(A_{n})\leq |A_{n}|+|B_{n}|$.
\end{enumerate}

Hence (if $E_{i}= \emptyset $ then $\mathfrak{o}(E_{i})=0$ otherwise using
the above inequalities)
\begin{align*}
\mathfrak{o}((B_{m}0^{k_{m}}){}_{[t,t+s]})= & \sum_{i=1}^{6}\mathfrak{o}%
(E_{i})+\max\{\mathfrak{o}(E_{7}),\mathfrak{o}(E_{8}),\mathfrak{o}(E_{9})\}
\\
\leq & \left( \frac{s+1}{|A_{n}|+2k_{n}+|B_{n}|}+1\right) \cdot
(|A_{n}|+|B_{n}|).
\end{align*}

If the extra assumption of $t$ does not hold we can continue the (eventually
finite) process of expanding $B_{m_{i}}$; this process will only provide
words of the same type as considered before.

This finishes the proof.
\end{proof}

\begin{lem}
\label{lem:Count-3} If $j-i>|A_{n}|,\ n\ge N$ and $x_{[i,i+|A_{n}|-1]}=A_{n}$%
, then
\begin{equation*}
\mathfrak{o}(x_{[i,j-1]})\leq\left( \frac{j-i}{|A_{n}|+2k_{n}+|B_{n}|}%
+2\right) (|A_{n}|+|B_{n}|).
\end{equation*}
\end{lem}

\begin{proof}
Observe that $A_{n}$ ($n\geq N$) does not appear in $y$ and $%
x_{[i,i+|A_{n}|-1]}$ is equal to $A_{n}$. We have to consider different
cases depending of the locations of $i$ and $j$:

\begin{enumerate}
\item If there exists $p\in \mathbb{N}$ and a sequence $\left\{
m_{i}\right\} _{i=1}^{p}$ with $m_{i}\geq n$ such that
\begin{equation*}
x_{[i,j-1]}=\sqcup
_{i=1}^{p-1}A_{n}0^{k_{m_{i}}}B_{m_{i}}0^{k_{m_{i}}}\sqcup
(A_{n}0^{k_{m_{p}}}B_{m_{p}}0^{k_{m_{p}}}){}_{[1,s_{1}]}
\end{equation*}%
%
%
%
%
%
for some $1\leq s_{1}<|A_{n}|+2k_{m_{p}}+|B_{m_{p}}|$. By Lemma \ref%
{lem:P-system-almostMeanEqui-1} it is not hard to see that
\begin{equation*}
\mathfrak{o}(x_{[i,j-1]})\leq \left( \frac{j-i}{|A_{n}|+2k_{n}+|B_{n}|}%
+1\right) (|A_{n}|+|B_{n}|).
\end{equation*}%
Hence the inequality follows.

\item If
there exist $m\geq n+1$ and $1\leq t\leq |B_{m}|$ such that the initial
position of $x_{[i,j-1]}$ is exactly the $t$-th position of $B_{m}$ and $%
(B_{m})_{[t,t+|A_{n}|-1]}=A_{n}$. Then two cases are involved:

\begin{enumerate}
\item[(a)] if $|A_{n}|<j-i\leq |B_{m}|+k_{m}-t+1$ then $x_{[i,j-1]}$ can be
expressed as
\begin{equation*}
x_{[i,j-1]}=(B_{m}0^{k_{m}}){}_{[t,t+s_{2}]}
\end{equation*}%
for some $0\leq s_{2}\leq |B_{m}|+k_{m}-t$. By Lemma \ref%
{lem:P-system-almostMeanEqui-3} we conclude this case.

\item[(b)] if $j-i>|B_{m}|+k_{m}-t+1$ then we can write $x_{[i,j-1]}$ as
\begin{equation*}
x_{[i,j-1]}=(B_{m}0^{k_{m}}){}_{[t,|B_{m}|+k_{m}]}\sqcup
_{i=1}^{q-1}A_{n}0^{k_{m_{i}}}B_{m_{i}}0^{k_{m_{i}}}\sqcup
(A_{n}0^{k_{m_{q}}}B_{m_{q}}0^{k_{m_{q}}}){}_{[1,s_{3}]}
\end{equation*}%
where $q\in \mathbb{N}$, $\left\{ m_{i}\right\} _{i=1}^{q}$ is a sequence
with $m_{i}\geq n,$ and $1\leq s_{3}<|A_{n}|+2k_{m_{q}}+|B_{m_{q}}|$.
Applying Lemmas \ref{lem:P-system-almostMeanEqui-1} ,~\ref{lem:P-system-almostMeanEqui-3} we finish this case and so the whole proof
is completed.
\end{enumerate}
\end{enumerate}
\end{proof}

Now we are going to prove Proposition \ref{prop:P-system}.

\begin{proof}[Proof of Proposition \protect\ref{prop:P-system}]

To show $X$ is almost mean equicontinuous we will show that the transitive
point $x$ is mean equicontinuous.

Let $\varepsilon >0$. There exists $K\in \mathbb{N}$ such that for any $%
a,b\in X$ we have that if $a_{[1,K]}=b_{[1,K]}$ then $d(a,b)<\varepsilon /5$%
. Since $k_{n}$ satisfies \eqref{cond:2}, i.e. $k_{n}\geq
n(2|A_{n}|+|B_{n}|) $, we conclude
\begin{equation*}
\lim_{n\rightarrow \infty }\frac{|A_{n}|+|B_{n}|}{|A_{n}|+|B_{n}|+2k_{n}}=0.
\end{equation*}%
Considering this there exists $m\in \mathbb{N}$ large enough such that
\begin{equation*}
\frac{2K(|A_{m}|+|B_{m}|)}{|A_{m}|+|B_{m}|+2k_{m}}<\frac{\varepsilon }{4}.
\end{equation*}

Let $z\in \lbrack A_{m}]$. If there exists a $k\in \mathbb{Z}_{+}$ such that
$\sigma ^{k}z=0^{\infty }$then
\begin{equation*}
\limsup_{n\rightarrow \infty }\frac{1}{n}\sum_{i=0}^{n-1}d(\sigma
^{i}z,0^{\infty })=0.
\end{equation*}

Otherwise $\sigma ^{k}z\neq 0^{\infty }$ for all $k\in \mathbb{Z}_{+}$. For
any $i\in \mathbb{N}$, there is an $l_{i}\in \mathbb{N}$ such that $%
z_{[1,1+i]}=x_{[l_{i},l_{i}+i]}$. Since $z$ starts with $A_{m}$, then by
Lemma~\ref{lem:Count-3} and the choice of $K$, we have that
\begin{align*}
\limsup_{n\rightarrow \infty }\frac{1}{n}\sum_{i=0}^{n-1}d(\sigma
^{i}z,0^{\infty })\leq & \frac{\varepsilon }{5}\limsup_{n\rightarrow \infty }%
\frac{1}{n}\#\{0\leq i\leq n-1\colon (\sigma ^{i}z)_{[1,K]}=0^{K}\} \\
& +\limsup_{n\rightarrow \infty }\frac{1}{n}\#\{0\leq i\leq n-1\colon
(\sigma ^{i}z)_{[1,K]}\neq 0^{K}\} \\
<& \frac{\varepsilon }{4}+\limsup_{n\rightarrow \infty }\frac{1}{n}\cdot
2K\cdot \mathfrak{o}(z_{[1,n]}) \\
<& \frac{\varepsilon }{4}+\limsup_{n\rightarrow \infty }\frac{1}{n}\cdot
2K\cdot \lbrack \frac{n}{|A_{m}|+2k_{m}+|B_{m}|}(|A_{m}|+|B_{m}|)] \\
& +\limsup_{n\rightarrow \infty }\frac{1}{n}\cdot 4K\cdot (|A_{m}|+|B_{m}|)
\\
<& \frac{\varepsilon }{4}+\frac{\varepsilon }{4}=\frac{\varepsilon }{2}.
\end{align*}%
This implies that
\begin{equation*}
\limsup_{n\rightarrow \infty }\frac{1}{n}\sum_{i=0}^{n-1}d(\sigma
^{i}x,\sigma ^{i}z)<\varepsilon ,
\end{equation*}%
for any $z\in \lbrack A_{m}]$ and hence $x$ is a mean equicontinuous point, completing the whole proof.
\end{proof}

Inspired by Corollary \ref{Dich-E-Cor}, the next we naturally ask if every $%
E $ (or $M$, $P$)-system is either Banach-mean/diam-mean sensitive or
Banach-mean/diam-mean equicontinuous. 
We conjecture that this question is still not true for $E$-system in
general, and we leave it open here. In the following we will show that for $%
M $-systems the dichotomy theorem is always true for Banach-mean notions.

\begin{thm}\label{thm:Banach-mean-equi-M-system}
Each Banach-mean equicontinuous $M$-system $(X,T)$ is minimal.
\end{thm}
\begin{proof}
If $X$ is a non-minimal $M$-system, by \cite[Theorem 4]{TZ11} there
exists a $\delta>0$, such that for any non-empty open subset $U\subset X$ there is a
pair $(x,y)\in U\times U$ satisfying the set $F=\{n\in{\mathbb{Z}}_{+}\colon
d(T^{n}x,T^{n}y)>\delta\}$ is a thick set. Note that the upper Banach
density of each thick set is exactly one, then similarly to Remark \ref{rem:equi}\eqref{rem:equi:3} we have
\begin{align*}
\limsup_{N-M\rightarrow\infty}\frac{1}{N-M}\sum_{i=M}^{N-1}d(T^{n}x,T^{n}y)
& >\limsup_{N-M\rightarrow\infty}\frac{1}{N-M}\sum_{i=M}^{N-1}(\delta
\#\{F\cap\lbrack M,N-1]\}) \\
& =\delta\cdot BD^{\ast}(F)=\delta.
\end{align*}
That is $(X,T)$ is Banach-mean sensitive, a contradiction.
\end{proof}

Combining \cite[Propsition 6.1(5)]{LTY13} and Theorem \ref{thm:Banach-mean-equi-M-system} we immediately have

\begin{cor}
\label{cor:Dich-M-system} Let $(X,T)$ be an $M$-system. Then it is either
Banach-mean sensitive or Banach-mean equicontinuous.
\end{cor}


\section{Conditions that imply mean sensitivity}

\label{sect:Positive-Cond}

In this section we are looking for sufficient conditions under which a
t.d.s. is mean sensitive. First, we focus on a special class of transitive
and ergodic systems. To begin with, we need some results and definitions
introduced in \cite{GM}.

Let $(X,T)$ be a t.d.s., $\mu $ be an ergodic measure on $X$ and $f\in
L^{2}(\mu )$. We say that $f$ is an \textit{almost periodic function} if $%
\overline{\{U_{T}^{j}f\colon j\in {\mathbb{Z}}_{+}\}}$ is a compact subset
of $L^{2}(\mu )$. We have that an ergodic system $(X,\mu ,T)$ is weakly
mixing if and only if every almost perioc function is constant. On the other
hand $(X,\mu ,T)$ has discrete spectrum if and only if every $f\in L^{2}(\mu
)$ is almost periodic (these two results are due to Halmos and Von-Neumann;
for detailed proofs see \cite{Wal82}).

Now assume $f$ is a continuous function on $X$. We say $(X,T)$ is \textit{$f$%
-mean sensitive} if there exists $\varepsilon >0$ such that for any
non-empty open subset $U\subset X$ we can find $x,y\in U$ satisfying
\begin{equation*}
d_{f}(x,y):=\lim_{n\rightarrow \infty }\frac{1}{n}\sum_{j=0}^{n-1}\left\vert
f(T^{j}x)-f(T^{j}y)\right\vert ^{2}d\mu >\varepsilon .
\end{equation*}%
When considering the opposite side, we say that $x\in X$ is an \textit{$f$%
-mean equicontinuous point} if for any $\varepsilon >0$ there is $\delta >0$
such that if $d(x,y)<\delta $ then $d_{f}(x,y)<\varepsilon $. $(X,T)$ is
said to be an \textit{almost $f$-mean equicontinuous system} if there is a
point which is both transitive and $f$-mean equicontinuous.
Similarly to the proof of \cite[Theorem 2.14]{GM} we have:

\begin{lem}
\label{lem:f-almost-mean-equi} Let $(X,T)$ be a t.d.s. Then $(X,T)$ is
almost mean equicontinuous if and only if it is almost $f$-mean
equicontinous for every continuous function $f$.
\end{lem}

\begin{lem}
\label{lem:non-almost-periodic-ergodic} Let $(X,T)$ be a transitive t.d.s., $%
\mu $ an ergodic measure on $X$ and $f\in L^{2}(\mu )$ a non-almost periodic
function. If for any non-empty open subset $A\subset X$ and any infinite
subset $S\subset {\mathbb{Z}}_{+}$, there are $s\neq t\in S$ and a generic
point $z\in X$ for $\mu $ (with respect to continuous functions) such that $%
z\in T^{-s}A\cap T^{-t}A$, then $(X,T)$ is mean sensitive.
\end{lem}

\begin{proof}
Continuous functions are dense in $L^{2}(X,\mu )$ and the set of non-almost
periodic functions is open (\cite[Theorem 1.13]{GM}), so there exists a
continuous function $f\in L^{2}(X,\mu )$ that is not almost periodic. This
implies it is not totally bounded, hence there exists $\varepsilon >0$ and
an infinite subset $S\subset \mathbb{Z}_{+}$ such that
\begin{equation*}
\int \left\vert U^{i}f-U^{j}f\right\vert ^{2}d\mu \geq \varepsilon \text{
for every }i\neq j\in S.
\end{equation*}%
For any non-empty open subset $A\subset X$ there exists $s\neq t\in S$ and a
generic point $z\in X$ for $\mu $ such that $z\in T^{-s}A\cap T^{-t}A$.
Define $g(x):=\left\vert U_{T}^{s}f(x)-U_{T}^{t}f(x)\right\vert ^{2}$. Since
$f$ is continuous, so is $g$. As $z$ is a generic point for $\mu $, then for
the continuous function $g$ we have
\begin{equation*}
\lim_{n\rightarrow \infty }\frac{1}{n}\sum_{j=0}^{n-1}g(T^{j}z)=\int g\ d\mu
\geq \varepsilon .
\end{equation*}%
Let $p:=T^{s}z\in A$ and $q:=T^{t}z\in A$. Then
\begin{align*}
\lim_{n\rightarrow \infty }\frac{1}{n}\sum_{j=0}^{n-1}\left\vert
f(T^{j}p)-f(T^{j}q)\right\vert ^{2}=& \lim_{n\rightarrow \infty }\frac{1}{n}%
\sum_{j=0}^{n-1}\left\vert f(T^{j+s}z)-f(T^{j+t}z)\right\vert ^{2} \\
=& \lim_{n\rightarrow \infty }\frac{1}{n}\sum_{j=0}^{n-1}g(T^{j}z)>%
\varepsilon .
\end{align*}%
That is, for any non-empty open subset $A\subset X$, there are $p,q\in A$
such that
\begin{equation*}
\lim_{n\rightarrow \infty }\frac{1}{n}\sum_{j=0}^{n-1}\left\vert
f(T^{j}p)-f(T^{j}q)\right\vert ^{2}>\varepsilon .
\end{equation*}%
Following the definition we know that $(X,T)$ is $f$-mean sensitive. By \cite%
[Theorem 2.13]{GM} $(X,T)$ is not almost $f$-mean equicontinous. And by
Lemma~\ref{lem:f-almost-mean-equi} $(X,T)$ is also not almost mean
equicontinuous. Now using dichotomy theorem in \cite{LTY13, GM14} we have
that $(X,T)$ is mean sensitive.
\end{proof}

In the following theorem, (2) and (3) are corollaries from previously known
results. We write them together for comparison.

\begin{thm}
\label{thm:meanSen-SuffCondition} Let $(X,T)$ be a transitive t.d.s. and $%
\mu $ be an ergodic measure on $X$. If one of the followings hold:

\begin{enumerate}
\item \label{thm:meanSen-SuffCondition:1} $(X,T)$ is uniquely ergodic,
topologically mixing and has positive topological entropy,

\item \label{thm:meanSen-SuffCondition:2} $\mu $ has full support and $%
(X,\mu ,T)$ does not have discrete spectrum (in particular if it is
non-trivial and measurably weakly mixing or has positive entropy with
respect to $\mu$),

\item \label{thm:meanSen-SuffCondition:4} $(X,T)$ is minimal and has
positive topological entropy,
\end{enumerate}

then $(X,T)$ is mean sensitive.
\end{thm}

\begin{proof}
\eqref{thm:meanSen-SuffCondition:1} Since $(X,T)$ is uniquely ergodic (with
the unique ergodic measure $\mu$) and has positive topological entropy, then
it has no discrete spectrum, and hence there is a non-almost periodic
function $f\in L^{2}(\mu)$ (see, e.g., \cite{Wal82, GM}). Since $(X,T)$ is
topologically strongly mixing, for any non-empty open subset $A\subset X$
and any infinite subset $S\subset {\mathbb{Z}}_{+}$, it is clear that $%
\{T^{n}A\colon n\in S\}$ is dense in $X$. This implies that there are $s\neq
t\in S$ such that $T^{-s}A\cap T^{-t}A\neq\emptyset$. By the unique
ergodicity of $(X,T)$, we know that each non-empty open subset $T^{-s}A\cap
T^{-t}A$ contains a generic point for $\mu $. Using Lemma~\ref%
{lem:non-almost-periodic-ergodic} then $(X,T)$ is mean sensitive.

\eqref{thm:meanSen-SuffCondition:2} Since $\mu $ has full support every open
set has positive measure. In \cite{GM14} it is shown that if $(X,\mu ,T)$
does not have discrete spectrum then there exists $\delta >0$ such that for
every set of positive measure $U\subset X$, there are $x,y\in U$ such that
\begin{equation*}
\limsup_{n\rightarrow \infty }\frac{1}{n}\sum_{i=0}^{n-1}d(T^{i}x,T^{i}y)>%
\delta .\
\end{equation*}

\eqref{thm:meanSen-SuffCondition:4} This result can be found in \cite{LTY13,
GM14}. Note that for any minimal t.d.s. $(X,T)$, each ergodic measure on $X$
has full support. Using variation principle and %
\eqref{thm:meanSen-SuffCondition:2} it is not hard to see that $(X,T)$ is
mean sensitive.
\end{proof}

What we can see is that positive entropy plus a strong form of topological
ergodicity implies mean sensitivity. We don't know if every transitive (or
even weakly mixing) uniquely ergodic t.d.s. with positive entropy is mean
sensitive. Note that without the transitivity condition, this question has a
negative answer.

\begin{thm}
\label{unpos}There exists a uniquely ergodic t.d.s. with positive entropy
that is not sensitive.
\end{thm}

\begin{proof}
Let $Y\subset\left\{ 0,1\right\} ^{\mathbb{N}}$ a minimal uniquely ergodic
subshift with positive topological entropy. Let $X=Y\cup\left\{ 2,3\right\}
^{\mathbb{N}},$ $y\in Y$ and $T:X\rightarrow X$ defined as follows
\begin{equation*}
Tx=\left\{
\begin{array}{cc}
\sigma(x), & \text{if }x\in Y, \\
y, & \text{otherwise.}%
\end{array}
\right.
\end{equation*}
It is not hard to see that $(X,T)$ is uniquely ergodic and that it is not
sensitive (using the open set $\left\{ 2,3\right\} ^{\mathbb{N}}$).
\end{proof}



The following can be deduced from results in \cite{LTY13}.

\begin{thm}
Every nontrivial minimal topologically weakly mixing t.d.s. is mean
sensitive.
\end{thm}

\begin{proof}
Let $(X,T)$ be a non-trivial minimal topological weakly mixing t.d.s. Assume
that it is not mean sensitive. Then from \cite[Corollary 5.5]{LTY13} or \cite%
[Theorem 8]{GM14} $(X,T)$ is mean equicontinuous.

Let $f_{eq}:X\rightarrow X_{eq}$ be the maximal equicontinuous factor map.
Then by \cite[Corollary 3.6]{LTY13} $f_{eq}$ is proximal, i.e., $%
f_{eq}(x)=f_{eq}(y)$ if and only if
\begin{equation*}
\liminf_{n\rightarrow\infty}d(T^{n}x,T^{n}y)=0.
\end{equation*}
Since $(X,T)$ is minimal and weakly mixing, we have that $X$ has no
non-trivial equicontinuous factors, i.e. $X_{eq}$ is a singleton. It then
follows that $X$ is proximal, contradicting the assumptions that $(X,T)$ is
non-trivial minimal. 
We conclude $(X,T)$ is mean sensitive.
\end{proof}

\medskip Now we will study systems with shadowing property. Recall that a
sequence $\{x_{n}\}_{n=0}^{\infty }$ is a \emph{$\delta $-pseudo-orbit} for $%
T$ if $d(x_{{n+1}},Tx_{n})<\delta $ for all $n\in {\mathbb{Z}_{+}}$, and is
\emph{$\epsilon $-traced} by a point $x\in X$ if $d(T^{n}x,x_{n})<\epsilon $
for all $n\in {\mathbb{Z}_{+}}$. We say that a t.d.s. $(X,T)$ has the \emph{%
shadowing property} if for any $\varepsilon >0$ we can find a $\delta >0$
such that each $\delta $-pseudo-orbit for $T$ is $\varepsilon $-traced by
some point of $X$.

It is known that a transitive system with shadowing property is either
equicontinuous or sensitive (e.g., \cite[Theorem 6]{Moo11}). Now we show
that if the transitive system with shadowing property has positive
topological entropy then, it is mean sensitive. This result is inspired by
\cite{LLT16}.


To prove it, we need the notions of sensitive and distal pair. A pair $%
(x_1,x_2)\in X^2\setminus\Delta_2$ is called a \emph{sensitive pair} \cite%
{YZ08} if for any neighbourhood $U_i$ of $x_i$ ($i=1,2$), and any non-empty
open subset $U$ of $X$ there exist $k\in\mathbb{N}$ and $y_i\in U$ such that
$T^k y_i\in U_i$ for $i=1,2$; and a \emph{distal pair} (e.g. \cite{Fur81})
if $\liminf_{n\to\infty} d(T^{n} x_{1}, T^{n} x_{2})>0$.

\begin{thm}
\label{thm:Tran+Shadowing+PosEntr=>MeanSen} Let $(X,T)$ be a transitive
t.d.s. with shadowing property. If $(X,T)$ has positive entropy, then it is
mean sensitive.
\end{thm}

\begin{proof}
Following from the proof of Lemma 3.3 in \cite{LLT16} we know that there
exists a sensitive and distal pair $(v_1,v_2)$ in $X^{2}$. 
By the definition of distal pair, there are $\varepsilon>0$ and $m\in\mathbb{%
N}$ such that
\begin{equation*}
d(T^{\ell} v_{1}, T^{\ell} v_{2})>3\varepsilon
\end{equation*}
for all $\ell\ge m$. 
For this $\varepsilon$, choose $\delta>0$ such that each $\delta$%
-pseudo-orbit can be $\varepsilon$-traced. Let $U, W$ be non-empty open
subsets of $X$ with $d(W, X\setminus U)>\varepsilon$, and $V_{i}$ be
neighbourhoods of $v_{i}$ with $\diam (V_{i})<\delta$ for $i=1,2$. Since $%
(v_{1},v_{2})$ is a sensitive pair, there are $x_{1}, x_{2}\in W$ and $n\in%
\mathbb{N}$ such that
\begin{equation*}
T^{n} x_{1}\in V_{1} \text{ and } T^{n} x_{2}\in V_{2}.
\end{equation*}
Fix $i\in\{1,2\}$. It is easy to see that the sequence
\begin{equation*}
x_{i}, Tx_{i}, \ldots, T^{n-1}x_{i}, v_{i}, T v_{i} \ldots
\end{equation*}
is a $\delta$-pseudo-orbit. Then there exists $y_{i}\in X$ such that
\begin{equation*}
d(x_{i}, y_{i})<\varepsilon\text{ and } d(T^{k-n} v_{i}, T^{k}
y_{i})<\varepsilon, \ k=n,n+1,\ldots.
\end{equation*}
This implies that $y_{i}\in U$ since $d(W, X\setminus U)>\varepsilon$, and
\begin{equation*}
d(T^{k} y_{1}, T^{k} y_{2})>\varepsilon
\end{equation*}
for all $k\ge m+n$. Hence
\begin{equation*}
\limsup_{n\to\infty}\frac{1}{N}\sum_{n=0}^{N-1} d(T^{n} y_{1}, T^{n}
y_{2})>\varepsilon,
\end{equation*}
showing that $(X,T)$ is mean sensitive. 
\end{proof}

%

As a consequence we have that

\begin{cor}
\label{cor:Trans+ShadowProp} Each transitive t.d.s. with shadowing property
is either mean equicontinuous or mean sensitive.
\end{cor}

\part{Applications}

In this part we show applications of our previous results. Corollary \ref%
{Dich-E-Cor} was already explained in Section \ref{sect:Proof-Of-Theorem 1.4}%
. The other applications need further explanations.

\section{Local mean equicontinuity}

\label{sect:LME}

In \cite{GW00} Glasner and Weiss introduced the concept of local
equicontinuity. A t.d.s. $(X,T)$ is \textit{locally equicontinuous }if for
every $x\in X$ we have that $\overline{\orb(x,T)}$ is almost equicontinuous.
Inspired by this we define local mean equicontinuity.

\begin{de}
A t.d.s. $(X,T)$ is \textit{locally mean equicontinuous }if for every $x\in
X $ we have that $\overline{\orb(x,T)}$ is almost mean equicontinuous.
\end{de}

In general locally equicontinuous systems have zero topological entropy.
This comes from the fact that almost equicontinuous systems have zero
topological entropy.

We have that almost mean equicontinuous systems may have positive
topological entropy (e.g. the example in Section \ref{sect:Proof-Of-Theorem
1.4}), nonetheless locally mean equicontinuous systems always have zero
topological entropy.

\begin{thm}
\label{thm:LME-0-entropy} Every locally mean equicontinuous t.d.s. has zero
topological entropy
\end{thm}

\begin{proof}
Assume $(X,T)$ has positve topological entropy. This implies there exists a
point $x\in X$ and an ergodic measure $\mu$ such that $\mu$ has positive
entropy and supp$(\mu)=\overline{\orb(x,T)}.$ By Theorem \ref%
{thm:meanSen-SuffCondition}\eqref{thm:meanSen-SuffCondition:2} we have that $%
\overline{\orb(x,T)}$ is mean sensitive. Hence $(X,T)$ is not locally mean
equicontinuous.
\end{proof}

Theorem 1.3 in \cite{GW00} says that if $(X,T)$ is locally equicontionuous
then every invariant ergodic probability measure on $X$ is supported on a
minimal subsystem . This is not true for locally mean equicontinuous systems
(as described in Theorem \ref{thm:LME}).

\begin{thm}
\label{thm:LME-nonMinimal} There exists a non-minimal E-system that is
locally mean equicontinuous.
\end{thm}

\begin{proof}
Consider the example constructed in Section \ref{sect:Proof-Of-Theorem 1.4}
with $(Y,\sigma)=\{0^\infty\}$. We already knew it is a non-minimal $E$%
-system (since it is Devaney chaotic as Proposition \ref{prop:Devaney}
stated); we will show it is locally mean equicontinuous. Let $x$ be the
transitive point constructed in Section \ref{sect:Proof-Of-Theorem 1.4} and $%
y\in X$. If $y\in \orb(x,\sigma)$ then $\overline{\orb(y,\sigma)}=X$ is
almost mean equicontinuous (this is proved in Section \ref%
{sect:Proof-Of-Theorem 1.4}). If $y\notin \orb(x,\sigma)$ then observing the
construction, $y$ is either a periodic point or a limit point with the form $%
B0^\infty$, where $B$ is a subword of some $A_n$ (furthermore the possible $%
B $ is the suffix of $A_n$). This implies that $\overline{\orb(y,\sigma)}$
is a finite set and hence almost mean equicontinuous, completing then the
whole proof.
\end{proof}

Now we are going to prove Theorem \ref{thm:LME}.

\begin{proof}[Proof of Theorem \protect\ref{thm:LME}]
The first part is Theorem \ref{thm:LME-0-entropy}. The second part follows
from Theorem \ref{thm:LME-nonMinimal}, since it is clear that $X=\supp (\mu)$
is not equal to the union of its minimal subsystems (in fact, $X$ is the
closure of the union of its minimal subsystems).
\end{proof}

\section{Hyperspaces}

\label{sect:Application-on-Hyperspace}

Let $X$ be a compact metric space. We define the hyperspace $K(X)$
as the space of non-empty closed subsets of $X$ equipped with the \textit{%
Hausdorff metric} $d_{H}$; which is defined by
\begin{align*}
d_{H}(A,B)& =\max \{\max_{x\in A}\min_{y\in B}d(x,y),\max_{y\in B}\min_{x\in
A}d(x,y)\} \\
& =\inf \{\varepsilon >0\colon B(A,\varepsilon )\supset B,\,B(B,\varepsilon
)\supset A\}
\end{align*}%
for $A,B\in K(X)$ and $B(A,\varepsilon )=\cup _{a\in A}B(a,\varepsilon )$.
The Hausdorff metric $d_{H}$ can induce a topology on $K(X)$ known as the
the \textit{Vietoris topology}.

Let $U_{1},\ldots ,U_{n}$ be non-empty open subsets of $X$ and fix $n\in
\mathbb{N}$. We have that
\begin{equation*}
\langle U_{1},\ldots ,U_{n}\rangle =\{A\in K(X):A\subset \cup _{i=1}^{n}U_{i}%
\mbox{ and }A\cap U_{i}\neq \emptyset \mbox{ for each }i=1,\ldots ,n\}.
\end{equation*}%
The following family
\begin{equation*}
\{\langle U_{1},\ldots ,U_{n}\rangle :U_{1},\ldots ,U_{n}%
\mbox{ are
non-empty open subsets of }X,n\in \mathbb{N}\}
\end{equation*}%
forms a basis for the Vietoris topology. With this topology $K(X)$ is
compact.

Given a continuous map $T:X\rightarrow X$ we induce a continuous map $%
T_{K}\colon K(X)\rightarrow K(X)$ by
\begin{equation*}
T_{K}(C)=TC\mbox{ for }C\in K(X).
\end{equation*}
We have that $(K(X),T_{K})$ is a t.d.s. We refer the reader not familiar
with hyperspace to \cite{Nad92} for more details.

\medskip For $n\in \mathbb{N}$ we define $K_{n}(X):=\{A\in K(X)\colon
|A|\leq n\}$ and $K_{\infty }(X):=\cup _{n\geq 1}K_{n}(X)$. The following
facts are easy to check.

\begin{lem}
\label{lem:Density-K(X)} Let $(X,T)$ be a t.d.s. Then

\begin{enumerate}
\item $K_{n}(X)$ is closed and $K_{\infty}(X)$ is dense in $K(X)$ (\cite[%
Lemma 2]{BS75});

\item if $Y$ is a dense subset of $X$, then $K_{\infty}(Y)$ is dense in $%
K(X) $.
\end{enumerate}
\end{lem}

\subsection{Proof of Theorem \protect\ref{thm:K(X)-Mean-Sen-X-mean-Equi}}

As mentioned before, every generalized diam-form of sensitivity can be
inherited from the hyperspace $K(X)$ to the phase space $X$ (see for example
\cite{WWC15}); nonetheless we will prove Theorem \ref%
{thm:K(X)-Mean-Sen-X-mean-Equi} which says this is not the case for (Banach)
mean sensitivity.

\begin{proof}[Proof of Theorem \protect\ref{thm:K(X)-Mean-Sen-X-mean-Equi}]
Let $(X,\sigma )$ be the example constructed in the proof of Theorem \ref%
{thm:Diam-Mean-Sen not Mean-Sen} ( in section \ref{sect:Proof-Of-Theorem 1.3}%
). By Theorem \ref{thm:Diam-Mean-Sen not Mean-Sen} we have that $(X,\sigma )$
is Banach-mean equicontinuous and hence it has zero topological entropy (see
for example \cite[Theorem 3.8]{LTY13}). Now we show that $(K(X),\sigma _{K})$
is mean sensitive.

Let $\mathcal{U}$ be a non-empty open subset of $K(X)$. Considering that $%
K_{\infty }(X) $ is dense in $K(X)$, $\orb(x,\sigma )$ is dense in $X$ and
Lemma~\ref{lem:Density-K(X)} we have that there is a point
\begin{equation*}
P=\{p_{1},\dots ,p_{m}\}\in K_{\infty }(\orb(x,\sigma ))\cap \mathcal{U}
\end{equation*}%
for some $m\in \mathbb{N}$. Let $\varepsilon >0$ be such that $%
B_{d_{H}}(P,\varepsilon )\subset \mathcal{U}.$ For each $1\leq k\leq m$
there exists $N_{k}>0$ such that $1/N_{k}<\varepsilon $. Since each $p_{k}$
lies on the orbit of $x$ under shift map we can see $(p_{k})_{[0,N_{k}]}$ as
subword for a certain $A_{n_{k}}$. Without loss of generality, we can assume
that the word $A_{n_{k}}$ ends with $(p_{k})_{[0,N_{k}]}$ (choose large
enough $N_{k}$ if necessary). By the construction, we have that each word $%
A_{i_{1}}$ ends with $A_{i_{2}}$ and appears in $B_{i_{3}}$ for any $%
i_{2}<i_{1}<i_{3}\in \mathbb{N}$.

Now we construct a closed subset $Q_{k}$ of $X$ as follows
\begin{align*}
Q_{k}=\{& (p_{k})_{[0,N_{k}]}10^{|A_{n_{k}}|-1}\dots , \\
& (p_{k})_{[0,N_{k}]}010^{|A_{n_{k}+1}|-2}\dots ,\cdots , \\
& (p_{k})_{[0,N_{k}]}0^{j}10^{|A_{n_{k}+j}|-(j+1)}\dots ,\cdots , \\
& (p_{k})_{[0,N_{k}]}0^{\infty }\}.
\end{align*}%
One can check that $d_{H}(\{p_{k}\},Q_{k})<\varepsilon $ and that for each $%
i>N_{k}$, there is a $q\in Q_{k}$ such that $(p_{k})_{i}\neq q_{i}$. Thus
\begin{equation*}
\limsup_{n\rightarrow \infty }\frac{1}{n}\sum_{i=0}^{n-1}d_{H}(\sigma
_{K}^{i}\{p_{k}\},\sigma _{K}^{i}Q_{k})=\limsup_{n\rightarrow \infty }\frac{1%
}{n}\sum_{i=N_{k}+1}^{n-1}1=1.
\end{equation*}

Let $Q:=\cup _{k=1}^{m}Q_{k}$ and $N:=\max_{1\leq k\leq m}N_{k}$. Clearly $%
d_{H}(P,Q)<\varepsilon $ and hence $Q\in \mathcal{U}$. By Lemma \ref%
{lem:Density-2} we have that $\overline{D}(E_{p_{k}})=0$ for each $k=1,\dots
,m$. This implies that if
\begin{equation*}
G=\{i\in \mathbb{N}\colon (p_{k})_{i}=0\mbox{ for all }k=1,\dots ,m\}
\end{equation*}%
then $\overline{D}(G)=1$. Note that for each $i\in G\cap (N,+\infty )$ we
can choose a $q^{\prime }\in Q$ with $q_{i}^{\prime }=1$. This implies that $%
d_{H}(\sigma _{K}^{i}P,\sigma _{K}^{i}Q)=1$, and hence we conclude that
\begin{align*}
\limsup_{n\rightarrow \infty }\frac{1}{n}\sum_{i=0}^{n-1}d_{H}(\sigma
_{K}^{i}P,\sigma _{K}^{i}Q)& \geq \limsup_{n\rightarrow \infty }\frac{1}{n}%
\sum_{i=N+1}^{n-1}d_{H}(\sigma _{K}^{i}P,\sigma _{K}^{i}Q) \\
& \geq \limsup_{n\rightarrow \infty }\frac{1}{n}\sum_{i\in \lbrack
N+1,n-1]\cap G}1 \\
& =\overline{D}(G)=1.
\end{align*}%
This shows $K(X)$ is mean sensitive.

It remains to show that the topological entropy of $K(X)$ is positive. By
the construction it is not hard to see that $(X,\sigma )$ contains a
subsystem $(Y,\sigma )$, which is the collection of points that contain at
most one $1$. That is, $Y=\{x^{i}\in \Sigma _{2}^{+}\colon i\in {\mathbb{Z}}%
_{+}\}$ with the shift map, where $x^{i}=x_{1}^{i}x_{2}^{i}\ldots
x_{j}^{i}\ldots $, $i\geq 0,j\geq 1$ with $x_{j}^{i}=1$ if $j=i,$ and $%
x_{j}^{i}=0$ if $j\neq i$. Since $(K(Y),\sigma _{K})$ has positive
topological entropy (see for instance \cite[Theorem 13]{KO07}), and $%
K(Y)\subset K(X),$ we conclude $(K(X),\sigma _{K})$ has positive topological
entropy.
\end{proof}

On the other hand we can not construct such counter-example for weakly mixing
or diam-mean sensitive systems (see Theorem \ref%
{thm:KK(X)-Mean-Sen=Diam-Mean-Sen}).

\medskip A question in a similar spirit is the following: Can a t.d.s. with
zero topological entropy have a hyperspace with uniform positive entropy
(u.p.e.) of order $n$ ($n\geq 2$)?

Let $n\geq 2$. We say $(x_{1},\ldots ,x_{n})\in X^{n}$ is an \textit{$n$%
-entropy tuple} if $\{x_{1},\ldots ,x_{n}\}$ is not a singleton and for each
pairwise disjoint closed neighbourhoods $V_{i}$ of $x_{i}$, $(X,T)$ has
positive topological entropy with respect to the open cover $\mathcal{U}%
=\{X\setminus V_{i}\colon 1\leq i\leq n\}$. We say that a t.d.s. $(X,T)$ has
\textit{uniform positive entropy (or is u.p.e.) of order $n$} if each $%
(x_{1},\ldots ,x_{n})\in X^{n}\setminus \Delta _{n}$ is an entropy tuple.

The answer to the question is no. It was shown in \cite[Theorem 8.4]{HY06}
that for any $n\geq 2$, $X$ is u.p.e. of order $n$ if and only if $K(X)$ is
u.p.e. of order $n$. We will generalize one of the implications of this
result in Theorem \ref{thm:K(X)-entrPair} by showing how we can find locally
the entropy pairs.

\begin{de}
\label{def:independence-set} Let $(X,T)$ be a t.d.s. and $\tilde{A}%
=(A_1,\ldots,A_k)$ be a tuple of subsets of $X$. We say that a subset $%
J\subseteq{\mathbb{Z}}_+$ is an independence set for $\tilde{A}$ (or that $%
\tilde{A}$ has the independence set $J$) if for any non-empty finite subset $%
I\subseteq J$, we have
\begin{equation*}
\bigcap_{i\in I}T^{-i}A_{s(i)}\not=\emptyset
\end{equation*}
for any $s\in \{1,\ldots,k\}^I$.
\end{de}

The following characterizations appeared in \cite{KL07}.

\begin{lem}
\label{lem:entropy-pair-Ind-Set} Let $(X,T)$ be a t.d.s. and $U_1,U_2$ be
pairwise disjoint subsets. Then

\begin{enumerate}
\item \label{lem:entropy-pair-Ind-Set:1} A pair $(x_1,x_2)\in X^2\setminus
\Delta_2$ is an entropy pair if and only if there is an independence set $S$
with positive density for $(U, V)$, where $U,V$ are any non-empty
neighbourhoods of $x_1,x_2$ respectively.

\item \label{lem:entropy-pair-Ind-Set:2} A pair $(U_1,U_2)$ has an
independence set with positive density if and only if $h_{\emph{top}%
}(T,\{U_{1}^{c},U_{2}^{c}\})>0$.

\item \label{lem:entropy-pair-Ind-Set:3} $h_{\emph{top}}(T,%
\{U_{1}^{c},U_{2}^{c}\})>0$ if and only if there is an entropy pair $%
(x_1,x_2)\in U_1\times U_2$.
\end{enumerate}
\end{lem}

\begin{thm}
\label{thm:K(X)-entrPair} Let $(X,T)$ be a t.d.s. and $A_{1},A_{2}\in K(X)$
with $A_{1}\cap A_{2}=\emptyset $. If $(A_{1},A_{2})$ is an entropy pair of $%
(K(X),\sigma _{K})$ then there exists an entropy pair, $(x_{1},x_{2})\in
A_{1}\times A_{2},$ of $(X,T)$.
\end{thm}

\begin{proof}
Let $n\in \mathbb{N}$. We define $U_{1}^{n}:=\cup _{a_{1}\in
A_{1}}B(a_{1},1/n)$ and $U_{2}^{n}:=\cup _{a_{2}\in A_{2}}B(a_{2},1/n)$.
Without loss of generality we can assume that $U_{1}^{1}\cap
U_{2}^{1}=\emptyset $. Since $(A_{1},A_{2})$ is an entropy pair, by Lemma %
\ref{lem:entropy-pair-Ind-Set}\eqref{lem:entropy-pair-Ind-Set:1} for each $%
n\in \mathbb{N}$, there is an independence set $S=S(n)$ with positive
density for $(\langle U_{1}^{n}\rangle ,\langle U_{2}^{n}\rangle )$. This
implies that for any finite subset $F\subset S$,
\begin{equation*}
\bigcap_{k\in F}T_{K}^{-k}\langle U_{i_{k}}^{n}\rangle \neq \emptyset
\end{equation*}%
for any $i_{k}\in \{1,2\}^{F}$. Let $A\in \bigcap_{k\in F}T_{K}^{-k}\langle
U_{i_{k}}^{n}\rangle $. We have that $T^{k}A\subset U_{i_{k}}^{n}$ for any $%
k\in F$, i.e. $S$ is also an independence set with positive density for $%
(U_{1}^{n},U_{2}^{n})$. By Lemma \ref{lem:entropy-pair-Ind-Set}%
\eqref{lem:entropy-pair-Ind-Set:2} we have that $h_{\text{top}%
}(T,\{(U_{1}^{n})^{c},(U_{2}^{n})^{c})\})>0$. Then using Lemma \ref%
{lem:entropy-pair-Ind-Set}\eqref{lem:entropy-pair-Ind-Set:3} there is an
entropy pair $(x_{1}^{n},x_{2}^{n})\in U_{1}^{n}\times U_{2}^{n}$. Assume
that $(x_{1}^{n},x_{2}^{n})\rightarrow (x_{1},x_{2})$ when $n\to\infty$
(passing to subsequence if necessary). It is easy to check that $%
(x_{1},x_{2})\in A_{1}\times A_{2}$ is an entropy pair.
\end{proof}

It is a natural question if every t.d.s. with u.p.e. of order $n$ $(n\in\N)$ is mean sensitive. This question will be answered negatively in a forthcoming paper of the second author by showing that there exists an almost mean equicontinuous t.d.s. which has u.p.e. of all orders \cite{Li17}.

\subsection{Relationships on hyperspace}

The notions of diam-mean equicontinuity and mean equicontinuity are
different in general. Actually in section \ref{sect:Proof-Of-Theorem 1.3} we
provided a transitive mean equicontinuous t.d.s. which has no diam-mean
equicontinuous points. Nonetheless on hyperspaces the notions coincide.

\begin{thm}
\label{thm:K(X)-Almo-MeanEqui=Almo-DiamMeanEqui} Let $(X,T)$ be a t.d.s. and
$A\in K(X)$. Then $A$ is diam-mean equicontinuous if and only if it is mean
equicontinuous.
\end{thm}

\begin{proof}
By definition every diam-mean equicontinuous point in $%
K(X) $ is mean equicontinuous.

Now assume that $A\in K(X)$ is mean equicontinuous. By Remark \ref{rem:equi}\eqref{rem:equi:2} for every $\varepsilon
>0 $ there is a $\delta >0$
such that
for any $B_{1},B_{2}\in B_{d_{H}}(A,\delta )$
we have that
$$
\limsup_{n\rightarrow \infty } \frac{1}{n}
\sum_{i=0}^{n-1}d_{H}(T_{K}^{i}B_{1},T_{K}^{i}B_{2}) <\frac{\varepsilon }{2}.
$$
Now we choose suitable $0<\delta ^{\prime }<\delta /2$ and an open subcover $%
\{B_{d}(x_{1},\delta ^{\prime }),\dots ,B_{d}(x_{m},\delta ^{\prime })\}$ of
$A$ such that
\begin{equation*}
\mathcal{V}=\langle B_{d}(x_{1},\delta ^{\prime }),\dots ,B_{d}(x_{m},\delta
^{\prime })\rangle \subset \langle \overline{B_{d}(x_{1},\delta ^{\prime })}%
,\dots ,\overline{B_{d}(x_{m},\delta ^{\prime })}\rangle \subset
B_{d_{H}}(A,\delta ).
\end{equation*}%
For simplicity we denote $\mathcal{W}=\langle \overline{B_{d}(x_{1},\delta
^{\prime })},\dots ,\overline{B_{d}(x_{m},\delta ^{\prime })}\rangle $. Now
set $E=\{x_{1},x_{2},\dots ,x_{m}\}$ and $F=\cup _{i=1}^{m}\overline{%
B_{d}(x_{i},\delta ^{\prime })}$. Clearly we have $E,\ F\in
B_{d_{H}}(A,\delta )$. Note that
\begin{align*}
\diam(T_{K}^{i}{\mathcal{V}})& \leq \sup_{B_{1}\in {\mathcal{W}}%
}d_{H}(T_{K}^{i}B_{1},T_{K}^{i}E)+\sup_{B_{2}\in {\mathcal{W}}%
}d_{H}(T_{K}^{i}E,T_{K}^{i}B_{2}) \\
& \leq
d_{H}(T_{K}^{i}F,T_{K}^{i}E)+d_{H}(T_{K}^{i}E,T_{K}^{i}F)=2d_{H}(T_{K}^{i}E,T_{K}^{i}F)
\end{align*}%
for all $i\in {\mathbb{Z}}_{+}.$ Thus
\begin{align*}
\limsup_{n\rightarrow \infty }\frac{1}{n}\sum_{i=0}^{n-1}\diam(T_{K}^{i}%
\mathcal{V})& \leq \limsup_{n\rightarrow \infty }\frac{1}{n}\sum_{i=0}^{n-1}%
\diam(T_{K}^{i}{\mathcal{W}}) \\
& \leq 2\limsup_{n\rightarrow \infty }\frac{1}{n}%
\sum_{i=0}^{n-1}d_{H}(T_{K}^{i}E,T_{K}^{i}F)<\varepsilon .
\end{align*}%
This implies that $A$ is a diam-mean equicontinuous point.
\end{proof}

As a consequence of Theorem \ref{thm:K(X)-Almo-MeanEqui=Almo-DiamMeanEqui}
we have the following corollary.

\begin{cor}
\label{cor:K(X)-Diam-MeanEqui=MeanEqui} Let $(X,T)$ be a t.d.s. Then $K(X)$
is diam-mean equicontinuous if and only if $K(X)$ is mean equicontinuous.
\end{cor}

Also we know that in general diam-mean sensitivity and mean sensitivity are
different concepts (e.g. see Theorem \ref{thm:Diam-Mean-Sen not Mean-Sen}).
We do not know if in general a hyperspace is mean sensitive if and only if
it is diam-mean sensitive. Nonetheless when combining \cite[Theorem 5.4]%
{LTY13}, \cite[Theorem 2]{Ban05} with Theorem \ref%
{thm:K(X)-Almo-MeanEqui=Almo-DiamMeanEqui} we obtain the following.

\begin{cor}
\label{cor:K(X)-Diam-MeanSen=MeanSen} \label{cor:WM} Let $(X,T)$ be a weakly
mixing t.d.s. Then $(K(X),T_{K})$ is diam-mean sensitive if and only if $%
(K(X),T_{K})$ is mean sensitive.
\end{cor}

We do not know if this characterization holds in general. Theorem %
\ref{thm:KK(X)-Mean-Sen=Diam-Mean-Sen} present a partial characterization.
%

Recall that a map $\pi \colon X\rightarrow Y$ is called
\textit{semi-open} if for every non-empty open subset $U\subset X$ the
interior of $\pi (U)$ in $Y$ is not empty. Let $x\in X$. We say that \textit{%
$\pi $ is open at $x$ }if for every neighbourhood $U$ of $x$, $\pi (U)$ is a
neighbourhood of $\pi (x)$. It is known that
$\pi $ is semi-open if and only if the set $\{x\in X\colon \pi
\mbox{ is open
at }x\}$ is residual \cite[Lemma 2.1]{Gla07}.

\begin{lem}
\label{lem:semi-open} Let $\phi\colon K(K(X))\to K(X)$ be defined by $%
\mathcal{A}\mapsto D=\cup_{A\in\mathcal{A}}A$. Then $\phi$ is a well-defined
semi-open factor map.
\end{lem}

\begin{proof}
To see that $\phi $ is well-defined, it suffices to check that $D=\cup
_{A\in \mathcal{A}}A\in K(X)$ for any given $\mathcal{A}\in K(K(X))$. Let $%
\{x_{n}\}_{n=1}^{\infty }$ be a sequence of $D$ and $x_{n}\rightarrow x$,
and $A_{n}\in \mathcal{A}$ with $x_{n}\in A_{n}$ for each $n\in \mathbb{N}$.
Assume $A_{n}$ converges to $A\in \mathcal{A}$ $\in K(K(X))$. This implies
that $x\in A\subset D$ and thus $D\in K(X)$. Furthermore, we observe that $%
d_{H}(\cup _{A\in \mathcal{A}}A,\cup _{B\in \mathcal{B}}B)\leq d_{H}(%
\mathcal{A},\mathcal{B})$ for any $\mathcal{A},\mathcal{B}\in K(K(X))$, $%
\phi (\{C\})=C$ for any $C\in K(X),$ and $T_{K}\circ \phi =\phi \circ T_{K}$%
. Hence $\phi $ is a factor map.

Now let $\mathcal{U}=\langle \mathcal{U}_{1},\mathcal{U}_{2},\ldots ,%
\mathcal{U}_{m}\rangle $ be a non-empty open subset of $K(K(X))$, where each
$\mathcal{U}_{j}$ is non-empty open in $K(X)$. Without loss of generality we
assume that for each $1\leq j\leq m$ there are non-empty open subsets $%
U_{1}^{j},\ldots ,U_{n}^{j}$ such that $\langle U_{1}^{j},\ldots
,U_{n}^{j}\rangle \subset \mathcal{U}_{j}$.
Let $\mathcal{W}=\langle U_{1}^{1},\ldots ,U_{n}^{1},\ldots
,U_{1}^{m},\ldots ,U_{n}^{m}\rangle $. It is not hard to see that $\mathcal{W%
}$ is non-empty open in $K(X)$. We now claim that $\mathcal{W}\subset \phi (%
\mathcal{U})$. Indeed, for any $E$ in $\mathcal{W}$ there are non-empty open
subsets $V_{i}^{j}$ of $X$ such that $V_{i}^{j}\subset \overline{V_{i}^{j}}%
\subset U_{i}^{j},i=1,\ldots ,n,j=1.\ldots ,m$ and $E\in \langle
V_{1}^{1},\ldots ,V_{n}^{1},\ldots ,V_{1}^{m},\ldots ,V_{n}^{m}\rangle $.
Let $E_{i}^{j}=E\cap \overline{V_{i}^{j}}$ and $E^{j}=\cup
_{i=1}^{n}E_{i}^{j}$. Clearly $E_{i}^{j}\subset U_{i}^{j}$, $E^{j}\in
\mathcal{U}_{j}$ and $E=\cup _{j=1}^{m}E^{j}$. We define $\mathcal{E}%
:=\{E^{1},\ldots ,E^{m}\}.$ We have that $\mathcal{E}\in \mathcal{U}$ and $%
\phi (\mathcal{E})=E$. This shows $\phi $ is semi-open.
\end{proof}

\begin{lem}
\label{lem:lift} Let $\pi\colon(X,T)\to(Y,S)$ be a semi-open factor map. If $%
(Y,S)$ is diam-mean sensitive then so is $(X,T)$.
\end{lem}

\begin{proof}
The map $\pi ^{-1}\colon Y\rightarrow K(X)$ is an upper semi-continuous
function, so it possesses a residual subset $Y_{0}\subset Y$ of continuous
points. Since $\pi $ is semi-open, by \cite[Lemma 2.1]{Gla07} $X_{0}=\pi
^{-1}(Y_{0})$ is a residual set of $X$. Now let $U$ be a non-empty open
subset in $X$. Let $x_{0}\in X_{0}\cap U$ and $\varepsilon >0$ be such that $%
B(x_{0},\varepsilon )\subset U$. By the definition of $X_{0}$ and the
semi-openness of $\pi $, there is a continuous point $y_{0}\in Y_{0}\cap
\text{int}(\pi (B(x_{0},\varepsilon )))$ for $\pi ^{-1}$ with $y_{0}=\pi
(x_{0})$ (where $\text{int}(\cdot )$ denotes the interior of a given set).
Let $\sigma >0$ be such that $B(y_{0},\sigma )\subset \text{int}(\pi
(B(x_{0},\varepsilon )))$. Since $(Y,S)$ is diam-mean sensitive, similarly to Remark \ref{rem:equi}\eqref{rem:equi:3} there is
a $t>0$ such that the upper density of $F:=\{i\in {\mathbb{Z}}_{+}\colon %
\diam(S^{i}B(y_{0},\sigma ))>2t\}$ is greater than $t$. For each $i\in F$ there
exists $y_{i}\in B(y_{0},\sigma )$ such that
\begin{equation*}
d(S^{i}y_{i},S^{i}y_0)>t.
\end{equation*}%
By the continuity of $S^{i}$, there is a $y_{i}{^{\prime }}\in Y_{0}\cap
B(y_{0},\sigma )$ such that
\begin{equation*}
d(S^{i}y_{i},S^{i}y_{i}{^{\prime }})<d(S^{i}y_{i},S^{i}y_0)-t,
\end{equation*}%
and hence
\begin{equation*}
d(S^{i}y_{i}{^{\prime }},S^{i}y_0)>t.
\end{equation*}%
Since $y_{i}^{\prime }\in Y_{0}\cap B(y_{0},\sigma )\subset Y_{0}\cap \pi
(B(x_{0},\varepsilon ))$, there is an $x_{i}\in B(x_{0},\varepsilon )$ with $%
y_{i}^{\prime }=\pi (x_{i})$. Then
\begin{equation*}
d(\pi T^{i}x_{i},\pi T^{i}x_{0})>t.
\end{equation*}%
By the continuity of $\pi $, there is $0<s<t$ such that
\begin{equation*}
d(T^{i}x_{i},T^{i}x_{0})>s,
\end{equation*}%
thus $F\subset \{i\in {\mathbb{Z}}_{+}\colon \diam(T^{i}B(x_{0},\varepsilon
))>s\}.$ Similarly to Remark \ref{rem:equi}\eqref{rem:equi:3} we have
\begin{align*}
\limsup_{n\rightarrow \infty }\frac{1}{n}\sum_{i=0}^{n-1}\diam(T^{i}U)& \geq
\limsup_{n\rightarrow \infty }\frac{1}{n}\sum_{i=0}^{n-1}\diam%
(T^{i}B(x_{0},\varepsilon )) \\
& \geq s\cdot \overline{D}(F)>st>0,
\end{align*}%
showing that $(X,T)$ is diam-mean sensitive.
\end{proof}


The following lemma is a particular case of \cite[Corollary 1]{WWC15}.

\begin{lem}
\label{lem:F-Sen} Let $(X,T)$ be a t.d.s.
If $(K(X),T_{K})$ is diam-mean sensitive then so is $(X,T)$.
\end{lem}

Now we are ready to show Theorem \ref{thm:KK(X)-Mean-Sen=Diam-Mean-Sen}.

\begin{proof}[Proof of Theorem \protect\ref{thm:KK(X)-Mean-Sen=Diam-Mean-Sen}%
]
The equivalence of $\eqref{thm:Mean-Sen=Diam-Mean-Sen:1}\Longleftrightarrow %
\eqref{thm:Mean-Sen=Diam-Mean-Sen:2}$ follows from Lemmas \ref{lem:semi-open}%
, \ref{lem:lift} and \ref{lem:F-Sen}. When $X$ is weakly mixing, we use
Corollary \ref{cor:WM} to conclude that $\eqref{thm:Mean-Sen=Diam-Mean-Sen:1}%
\Longleftrightarrow \eqref{thm:Mean-Sen=Diam-Mean-Sen:2}\Longleftrightarrow %
\eqref{thm:Mean-Sen=Diam-Mean-Sen:3}\Longleftrightarrow %
\eqref{thm:Mean-Sen=Diam-Mean-Sen:4}$.
\end{proof}

\providecommand{\bysame}{\leavevmode\hbox to3em{\hrulefill}\thinspace} %
\providecommand{\MR}{\relax\ifhmode\unskip\space\fi MR }
\providecommand{\MRhref}[2]{  \href{http://www.ams.org/mathscinet-getitem?mr=#1}{#2}
} \providecommand{\href}[2]{#2}

\end{document}